%% file: morphismsarxiv.tex
\documentclass[final,12pt,a4paper]{amsart}

\input{header.tex}

\begin{document}

\title[Morphisms between indecomposable complexes]{Morphisms between
  indecomposable complexes in the bounded derived category of a gentle
  algebra}

\author{Kristin Krogh Arnesen}
\address{Institutt for Matematiske Fag, Norges Teknisk-Naturvitenskapelige Universitet, N-7491 Trondheim, Norway.}
\email{kristin.arnesen@math.ntnu.no}

\author{Rosanna Laking}
\address{School of Mathematics, The University of Manchester, Oxford Road, Manchester, M13 9PL, United Kingdom.}
\email{rosanna.laking@manchester.ac.uk}

\author{David Pauksztello}
\address{School of Mathematics, The University of Manchester, Oxford Road, Manchester, M13 9PL, United Kingdom.}
\email{david.pauksztello@manchester.ac.uk}



\begin{abstract}
In this article we provide a simple combinatorial description of morphisms between indecomposable complexes in the bounded derived category of  a gentle algebra. 
\end{abstract}

\maketitle

{\small
\setcounter{tocdepth}{1}
\tableofcontents
}

\addtocontents{toc}{\protect{\setcounter{tocdepth}{-1}}}  
\section*{Introduction} 
\addtocontents{toc}{\protect{\setcounter{tocdepth}{1}}}   

Triangulated categories are of central importance in many branches of
mathematics, providing a common framework for algebraists, geometers,
topologists and theoretical physicists, amongst others. Perhaps the
most famous illustration of their utility is Beilinson's equivalences
between the derived categories of coherent sheaves on projective
spaces and representations of certain finite-dimensional algebras
\cite{Beilinson}, which provided deep connections between algebra and
geometry.

In algebra and geometry, the triangulated categories we have in mind
are derived categories and categories constructed from them, for
example, cluster categories. However, despite their utility, there is
a major drawback: the construction of derived categories is abstract
and explicit computation is often difficult. Indeed, much intuition is
often obtained from examples of homological dimension one (=
hereditary), owing to particularly nice homological properties which
allow one to reduce computations to the (well-understood) abelian
categories with which one starts.

Developing intuition in such an abstract setting requires a good
collection of examples, for which computation becomes straightforward
and non-trivial phenomena can be observed concretely. In this article,
we shall show that so-called \emph{gentle algebras} provide a wide
class of such examples. Moreover, given their central place in the
current thrust of research in cluster-tilting theory, where they occur
as surface algebras \cite{ABCP}, concrete understanding of the derived
categories of gentle algebras is both useful and timely.

\subsubsection*{Main results}

A principal way in which one can understand the structure of a
category is, firstly, to describe all its indecomposable objects and,
secondly, the morphisms between them. This is demonstrated very
successfully by Auslander--Reiten (AR) theory.

Let $\Lambda$ be a gentle algebra and let $\Db(\Lambda)$ be its
bounded derived category with shift functor $\Sigma$.
The first step was accomplished by Bekkert and Merklen in \cite{BM}
who, inspired by a classic paper \cite{BR}, described the
indecomposable objects of $\Db(\Lambda)$ by string combinatorics:
these include, in the terminology of \cite{Bo}, the so-called
\emph{homotopy string complexes} and (one-dimensional) \emph{homotopy
  band complexes}, i.e. complexes which can be unfolded to look like
an oriented copy of a Dynkin diagram of type $A$ or type
$\widetilde{A}$, respectively, and whose differentials are paths in
the quiver of $\Lambda$. Thus, a homotopy string or band is none other
than a word whose letters consist of paths in the quiver of $\Lambda$
and their inverses.

In this article, we describe all morphisms between indecomposable
complexes in $\Db(\Lambda)$, thus completing the hands-on
combinatorial framework that facilitates straightforward computation
in these non-trivial categories. The main theorem is as follows.

\begin{introtheorem} \label{introthm:A} Let $X$ and $Y$ be homotopy
  string or one-dimensional band complexes in $\Db(\Lambda)$. Let
  $w_X$ and $w_Y$ be the words corresponding to $X$ and $Y$.  There is
  a canonical basis of $\Hom_{\Db(\Lambda)}(X,Y)$ given by the
  following three classes of maps:
\begin{itemize}
\item \textbf{\textit{graph maps:}} corresponding to the maximal
  overlaps in $w_X$ and $w_Y$ satisfying certain compatibility
  conditions at the endpoints.
\item \textbf{\textit{quasi-graph maps:}} corresponding to maximal
  overlaps in $w_X$ and $w_{\Sigma^{-1}Y}$ satisfying certain
  non-degeneracy conditions at the endpoints; these give rise to
  homotopy classes of maps $X \to Y$.
\item \textbf{\textit{singleton maps:}} certain special maps which can
  be detected easily from the word combinatorics of $w_X$ and $w_Y$.
\end{itemize}
\end{introtheorem}

The remaining indecomposable complexes in $\Db(\Lambda)$ are
constructed from homotopy band complexes. Homotopy band complexes sit
at the mouths of homogeneous tubes: a tube is indexed by a homotopy
band and a non-zero scalar. The object of length $n$ in a given tube
is specified by the additional data of an $n$-dimensional vector
space. We call these `higher-dimensional' band complexes; we refer to
Section~\ref{sec:bands} for more precise
details. Theorem~\ref{introthm:A} deals with one-dimensional band
complexes and string complexes, which are always
`one-dimensional'. What about maps involving higher dimensional band
complexes? Our second main theorem tells us that we don't have to
worry about them:

\begin{introtheorem} \label{introthm:B} Suppose $X$ and $Y$ are string
  or one-dimensional band complexes in $\Db(\Lambda)$. For $X$ a
  homotopy band, let $X_n$ be its `$n$-dimensional' version, otherwise
  take $n=1$ and $X_1 = X$; similarly for $Y$. Then, generically,
  $\dim \Hom(X_m,Y_n) = mn \cdot \dim \Hom(X,Y)$.
\end{introtheorem}

In Theorem~\ref{introthm:B} care must be taken when $X \cong Y$ or $X
\cong \Sigma^{-1}Y$; see Section~\ref{sec:bands} for precise
details. The module category analogues of these results are classical;
see \cite{C-B,C-B2,Krause}.

In summary: Theorems~\ref{introthm:A} and \ref{introthm:B} reduce some
difficult homological algebra to elementary word
combinatorics. Combined with the description of indecomposables, this
opens up a wide and natural class of examples of triangulated
categories to explicit computation.

\subsubsection*{Applications and context}
Before continuing, we make some remarks on the potential applications
of these results.  Gentle algebras present us with particularly good
candidates to begin a systematic `hands on' study of derived
categories for a number of reasons:
\begin{itemize}
\item Certain general aspects of the structure of $\Db(\Lambda)$ are
  already known: for instance, the Avella Alaminos--Gei\ss\ invariant
  describes certain fractionally Calabi-Yau triangulated subcategories
  whose AR components have a boundary \cite{Avella-Geiss}.
\item The AR theory of $\Db(\Lambda)$ can be computed: using the
  Happel functor $F \colon \DbL \into \stmod{\rL}$ \cite{Happel},
  where $\rL$ is the repetitive algebra of $\Lambda$, Bobi\'nski
  \cite{Bo} gave an algorithm which computes the AR triangles in
  $\DbL$. Indeed, as an application of our results, we recover
  Bobi\'nski's algorithm without recourse to the Happel functor, and
  the often unpleasant computations that ensue. Examples of the kinds
  of results we have in mind are the structural results on the AR
  quiver in \cite{Babaei,AG}, which use the combinatorial invariants
  of \cite{Bastian} to parametrise the AR components completely. Can
  one use these to extend the results of Avella Alaminos and Gei\ss\
  to get better derived invariants?
\item Vossieck \cite{Vossieck} introduced the family of
  \emph{derived-discrete algebras}, for which we now understand
  various non-trivial homological properties \cite{Bo2, BGS, BK, BPP1,
    BPP2}. These algebras are gentle, and thus they can be used as a
  template for further study of derived categories of gentle
  algebras. Here as an application of our results, we recover the
  universal Hom-dimension bound of \cite{BPP1}.
\end{itemize}
 
We expect our results to be useful in the classification of tilting
and, more generally, silting objects for gentle algebras, particularly
for surface algebras; such objects are very closely related to
cluster-tilting objects. We therefore expect that the string
combinatorics here will be adapted to the cluster combinatorics.

Note also that string combinatorics seem not to be confined to
representation theory and cluster theory: such combinatorics occur in
the setting of homological mirror symmetry. For example, analogous
descriptions of indecomposables in categories occurring in this
context were presented by R. Bocklandt in the talk `From A to B via
SYZ' at the workshop `Homological Interactions between Representation
Theory and Singularity Theory' held in Edinburgh in December 2014.

\subsection*{Acknowlegments}

We would like to thank Peter J\o rgensen and Mike Prest for valuable
comments. KA and DP gratefully acknowledge financial support of the
EPSRC through grant EP/K022490/1, and KA would like to thank The
University of Manchester for the kind hospitality during two research
visits. RL and DP also acknowledge the kind hospitality of the Algebra
Group at the Norwegian University of Science and Technology in
Trondheim, and for financial support on a research visit there.

\section{Preliminaries and notation} \label{sec:gentle}

Let $\Gamma = (\Gamma_0,\Gamma_1)$ be a finite connected quiver. 
Recall from \cite{AS} that a
bound path algebra $\Lambda \cong \kk \Gamma/I$ is called \emph{gentle} if:
\begin{enumerate}[label=(\arabic*)]
\item for each vertex $x$ of $\Gamma$, there are at most two arrows
  starting at $x$ and at most two arrows ending at $x$;
\item for any arrow $a$ in $\Gamma$ there is at most one arrow
  $b$ in $\Gamma$ such that $ab \notin I$ and at most one
  arrow $c$ in $\Gamma$ such that $ca \notin I$;
\item for any arrow $a$ in $\Gamma$ there is at most one arrow
  $b$ in $\Gamma$ such that $ab \in I$ and at most one
  arrow $c$ in $\Gamma$ such that $ca \in I$;
\item the ideal $I$ is generated by paths of length $2$. \label{item:length2}
\end{enumerate}

Let $P(x)$ be the indecomposable projective left $\Lambda$-module
corresponding to $x \in \Gamma_0$.  We recall the following useful
property of gentle algebras; see, for instance \cite[Section 3]{BM}.

\begin{proposition}\label{prop: paths}
There is a bijection
\begin{eqnarray*}
\{\text{paths } p\colon x \path y \text{ in } \Gamma\} & \bij & \{\text{basis elements of } \Hom_{\Lambda}(P(y),P(x))\}; \\
p & \longmapsto & (u \mapsto up).
\end{eqnarray*}
\end{proposition}

\begin{convention} \label{con:maps} From now on, by abuse of notation,
  we shall identify a path $p\colon x \path y$ with its corresponding
  basis element in $\Hom_{\Lambda}(P(y),P(x))$.  
\end{convention}

Throughout this article, we shall fix a gentle algebra $\Lambda = \kk
\Gamma / I$ over an algebraically closed field $\kk$. Algebraic
closure of $\kk$ is not strictly necessary, but it significantly
simplifies the presentation of the combinatorics.

All modules in this paper will be left modules. We shall be interested
in three categories:
\begin{itemize}
\item $\sC\coloneqq \CL$: the category of right bounded complexes of
  finitely generated projective $\Lambda$-modules whose cohomology is
  bounded;
\item $\sK\coloneqq \KbpL$: the homotopy category of bounded complexes
  of finitely generated projective $\Lambda$-modules -- the so-called
  \emph{perfect complexes};
\item $\sD\coloneqq \DbL = \Db(\mod{\Lambda})$: the bounded derived
  category of finitely generated $\Lambda$-modules.
\end{itemize}
Throughout the paper, we shall identify the bounded derived category
$\sD$ with the triangle equivalent category $\KpL$, consisting of
right bounded complexes of finitely generated projective
$\Lambda$-modules whose cohomology is bounded. This identification
allows us to use the combinatorics of homotopy strings and bands,
which will be described in the next section, throughout the paper.  We
direct the reader to consult \cite{Happel} for background on derived
and homotopy categories.

\section{Indecomposable objects in $\sD$} \label{sec:indec}

In this section we give an overview of Bekkert and Merklen's
description of the indecomposable objects in $\sD$.  Their crucial
observation is that it is enough to consider complexes where the
differential is given by matrices whose entries are either zero or a
path (cf. Convention \ref{con:maps}).  The indecomposable objects are
obtained by unravelling the differential into `homotopy strings and
bands' corresponding to perfect complexes, and `infinite homotopy
strings' for the unbounded complexes.
The reader is encouraged to have the following example in mind when
reading this section:

\begin{runningexample}\label{ex:first}
  Let $\Lambda = \kk \Gamma/I$ be given by the quiver
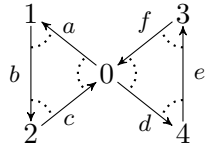
\[
\begin{tikzpicture}
  \node (0) at (0,0) [smallvertex] {$0$};
  \node (1) at (-1,0.8) [smallvertex] {$1$};
  \node (2) at (-1,-0.8) [smallvertex] {$2$};
  \node (3) at (1,0.8) [smallvertex] {$3$};
  \node (4) at (1,-0.8) [smallvertex] {$4$};
  \draw [->] (0) -- node [above, tinyvertex] {$a$} (1);
  \draw [->] (1) -- node [left, tinyvertex] {$b$} (2);
  \draw [->] (2) -- node [below, tinyvertex] {$c$} (0);
  \draw [->] (0) -- node [below, tinyvertex] {$d$} (4);
  \draw [->] (4) -- node [right, tinyvertex] {$e$} (3);
  \draw [->] (3) -- node [above, tinyvertex] {$f$} (0);
  \draw[dotted,thick] (-0.3,0.2) arc (135:225:8pt);
  \draw[dotted,thick] (0.3,0.2) arc (45:-45:8pt);
  \draw[dotted,thick] (-1,0.35) arc (-90:-20:8pt);
  \draw[dotted,thick] (-1,-0.35) arc (90:20:8pt);
  \draw[dotted,thick] (1,0.35) arc (-90:-160:8pt);
  \draw[dotted,thick] (1,-0.35) arc (90:160:8pt);
\end{tikzpicture}
\] 
  
\medskip
\noindent{The following complex, with the leftmost non-zero term in
  cohomological degree $0$, is an indecomposable object of $\sK$:}
\[
\xymatrix{
  0 \ar[r]
  & P(0) \ar[r]^-{\begin{pmat}c & f\end{pmat}}
  & P(2) \oplus P(3) \ar[r]^-{\begin{pmat}b & 0 \\ 0 & e\end{pmat}}
  & P(1) \oplus P(4) \ar[r]^-{\begin{pmat}af \\ 0\end{pmat}}
  & P(3) \ar[r] & 0. \\
}
\]
Observe that the criterion $d^2 = 0$ is obtained by either passing
through a relation, or by having $0$ in the differential.  We notice
that this complex `unfolds' as
\[
\xymatrix@R=0.4pc{
  2 & 1 & 0 & 1 & 2 & 3 \\
  P(4) & P(3) \ar[l]_{e} & P(0) \ar[l]_{f} \ar[r]^{c} & P(2)
  \ar[r]^{b} & P(1) \ar[r]^{af}& P(3) \\
}
\]
where the cohomological degrees are written above each module.
Moreover, the modules appearing are uniquely determined by the
endpoints of the maps, so all information in this complex is
communicated by the diagram
\[
\xymatrix@R=0.4pc{
  2 & 1 & 0 & 1 & 2 & 3 \phantom{\, .} \\
  \xydot & \xydot \ar[l]_{e} & \xydot \ar[l]_{f} \ar[r]^{c} & \xydot
  \ar[r]^{b} & \xydot \ar[r]^{af}& \xydot \, .
}
\]
This is what we will later define as a `homotopy string'.  Another way
of encoding this object is as a `string with degrees'
$(e,2,1)(f,1,0)(c,0,1)(b,1,2)(af,2,3)$.
\end{runningexample}

\begin{remark}\label{Rmk: transpose}  By using transposed matrices, matrix multiplication
(i.e. composing maps) fits with composition of paths.  For instance;
$\begin{pmat}c &f \end{pmat}\begin{pmat}b & 0 \\ 0 & e\end{pmat}
= \begin{pmat}cb & fe\end{pmat} = 0$.
\end{remark}

\subsection{Homotopy strings} \label{sec:strings}

A \emph{homotopy letter} is a triple $(p,i,j)$ where $p$ is a path in
$\Gamma$ with no subpath in $I$, and $i$, $j$ are integers such that
$|i-j| \leq 1$.  The homotopy letter $(p,i,j)$ is called \emph{direct}
if $i<j$ and \emph{inverse} if $i>j$.  If $i=j$ then $p$ must be a
stationary path and is called a \emph{trivial} homotopy letter. We set
$(p,i,j)^{-1} \coloneqq (p,j,i)$.  The starting and ending vertices of
a homotopy letter are defined as
\[
s(p,i,j) = \begin{cases}
  t(p) & \text{if $(p,i,j)$ is inverse;}\\
  s(p) & \text{otherwise;}\end{cases}
\quad \textrm{and} \quad
t(p,i,j) = \begin{cases}
  s(p) & \text{if $(p,i,j)$ is inverse; }\\
  t(p) & \text{otherwise.}\end{cases}
\]

The composition $(p,i,j)(p',i',j')$ is defined if $j = i'$ and
$s(p,i,j) = t(p',i',j')$.  If $(p,i,j)$ is trivial in this situation,
we write $(p,i,j)(p',i',j') \coloneqq (p',i',j')$; similarly if
$(p',i',j')$ is trivial then $(p,i,j)(p',i',j') \coloneqq (p,i,j)$.

\begin{convention}
  We shall often write $p$ as shorthand for $(p,i,j)$. However, the
  degrees $i,j$ should always be considered to be implicitly present.
\end{convention}

A \emph{homotopy string} is a sequence of pairwise
composable homotopy letters
\[
w = \prod_{r=n}^{1} (w_r,i_r,j_r) = (w_n,i_n,j_n)(w_{n-1},i_{n-1},j_{n-1})\cdots(w_2,i_2,j_2)(w_1,i_1,j_1)
\]
such that the following holds:
\begin{enumerate}[label=(\arabic*)]
  \item whenever $(w_r,i,i+1)(w_{r-1},i+1,i+2)$
    occurs, $w_rw_{r-1}$ has a subpath
    in $I$;
  \item whenever $(w_r,i,i-1)(w_{r-1},i-1,i-2)$ occurs, $w_{r-1}w_r$
    has a subpath in $I$;
  \item whenever $(w_r,i,i+1)(w_{r-1},i+1,i)$ occurs, $w_r$ and
    $w_{r-1}$ do not start with the same arrow;
  \item whenever $(w_r,i,i-1)(w_{r-1},i-1,i)$ occurs, $w_r$ and
    $w_{r-1}$ do not end with the same arrow.
\end{enumerate}
Write $s(w) = s(w_1,i_1,j_1)$ and $t(
w) = t(w_n,i_n,j_n)$ if $w$ is non-trivial, and for a trivial homotopy
string $w = (1_x,i,i)$ we write $s(w) = x = t(w)$.  Its inverse is given by
\[
((w_n,i_n,j_n)\cdots(w_1,i_1,j_1))^{-1} \coloneqq
(w_1,j_1,i_1) \cdots (w_n,j_n,i_n) \,.
\]

Another way of expressing a homotopy string $w =
\prod_{r=n}^1(w_r,i_r,j_r)$ is by a diagram
\[
\xymatrix@R=1mm@C=13mm{
i_n & j_n & & j_2 & j_1 \\
\xydot \ar@{-}[r]^{ w_n} & \xydot \ar@{-}[r] & \cdots
\ar@{-}[r] & \xydot \ar@{-}[r]^{ w_1} & \xydot,\\
}
\]
where the line labelled $w_r$ is an arrow pointing to the right if
$w_r$ is direct, and an arrow pointing to the left if $w_r$ is
inverse.  Note that each vertex
$\xymatrix{\ar@{-}[r]^{w_r}&\xydot\ar@{-}[r]^{w_{r-1}}&}$ corresponds
to a unique indecomposable projective $\Lambda$-module, namely $P(s(w_r)) =
P(t(w_{r-1}))$.

Let $w = \prod_{r=n}^1(w_r,i_r,j_r)$ be a homotopy string.  The
\emph{string complex} $P_w$ corresponding to $w$ is constructed as
follows.  We define indexing sets
\[
\mcI_i \coloneqq \begin{cases}\{r ~|~ ( w_r, i, i\pm 1) \text{ is in $ w$} \}
\sqcup \{0\} & \text{if $i = j_1$};\\
\{r ~|~ ( w_r, i, i\pm 1) \text{ is in $ w$} \} &\text{otherwise.}
\end{cases}
\]
Sitting in degree $i$, the corresponding complex has the object
$P^i_w$ given by
\[
\bigoplus_{r \in \mcI_i} P(\phi_w(r)), \quad \text{where $\phi_w(r) =
  t(w_r)$ for $r > 0$ and $\phi_w(0) = s(w_1)$.}
\]
The differentials are defined componentwise: Any direct homotopy
letter $w_r$ yields the component $P(\phi_w(r))
\stackrel{w_r}{\longrightarrow} P(\phi_w(r-1))$, and any inverse
homotopy letter $w_r$ yields the component $P(\phi_w(r-1))
\stackrel{w_r}{\longrightarrow} P(\phi_w(r))$.  Note that $P_{w} \cong
P_{w^{-1}}$.

\subsection{Homotopy bands}\label{sec:hom-bands}

A non-trivial homotopy string $w = \prod_{r=n}^{1}( w_r,i_r,j_r)$ is a
\emph{homotopy band} if $s(w) = t(w)$, $i_n = j_1$, one of
$\{w_n,w_1\}$ is direct and the other inverse, and $w$ is not a proper
power of another homotopy string.

We now describe how to construct \emph{one-dimensional band
  complexes}.  The higher dimen\-sional band complexes will be studied
in Section \ref{sec:bands}, and the definition is given there.  Fix a
homotopy band $w$ and an element $\lambda \in \kk^*$.  Again we define
indexing sets
\[
\mcI_i \coloneqq \{r ~|~ ( w_r, i, i\pm 1) \text{ is in $ w$} \} \, ,
\]
and $B_{w,\lambda,1}^i$ is defined as for string complexes, that is,
\[
B_{w,\lambda,1}^i = \bigoplus_{r \in \mcI_i} P(\phi_w(r))\quad
\text{where $\phi_w(r)$ is as before.}
\]
For band complexes with $n \geq 2$, the components of the differential
are determined as for string complexes, except for that corresponding
to $w_1$, which acquires the scalar $\lambda$: $P(\phi_w(1))
\stackrel{\lambda w_1}{\longrightarrow} P(\phi_w(n))$ for $w_1$
direct, and $P(\phi_w(n)) \stackrel{\lambda w_1}{\longrightarrow}
P(\phi_w(1))$ for $w_1$ inverse.  When $n=2$, the only non-zero
component of the differential of $B_{w,\lambda,1}$ is $\lambda w_1 +
w_2$.  Note that $B_{w,\lambda,1} \cong
B_{w^{-1},\frac{1}{\lambda},1}$.

The diagram notation of a homotopy band is an infinite repeating
diagram, with the scalar $\lambda$ attached to $w_1$ as in the
complex:
\[
\xymatrix@R=1mm@C=8mm{
& i_n & j_n & & j_2 & j_1=i_n & \\
\ar@{-}[r]^-{\lambda w_1} \cdots &\xydot \ar@{-}[r]^-{w_n} & \xydot \ar@{-}[r] & \cdots
\ar@{-}[r] & \xydot \ar@{-}[r]^-{\lambda w_1} & \xydot
\ar@{-}[r]^-{w_n} & {\cdots}.\\
}
\]

\begin{example}\label{ex:band}
  Consider the Running Example.  The homotopy band
  \[
  z = (d,3,2)(e,2,1)(f,1,0)(c,0,1)(b,1,2)(a,2,3) \, ,
  \] 
  together with a fixed scalar $\lambda$, gives rise to the complex
  $B_{z,\lambda,1}$, where again the leftmost component lies in degree
  0,
  \[
  \xymatrix{
    0 \ar[r]
    & P(0) \ar[r]^-{\begin{pmat}c & f\end{pmat}}
    & P(2) \oplus P(3) \ar[r]^-{\begin{pmat}b & 0 \\ 0 & e\end{pmat}}
    & P(1) \oplus P(4) \ar[r]^-{\begin{pmat}\lambda a \\ d\end{pmat}}
    & P(0) \ar[r] & 0. \\
  }
  \]
  We write this as the diagram
  \[
  \xymatrix@R=0.4pc{
    & 3 & 2 & 1 & 0 & 1 & 2 & 3 & \\
    \cdots \ar[r]^{\lambda a}& \xydot & \xydot \ar[l]_{d} & \xydot \ar[l]_{e}
    & \xydot \ar[l]_{f} \ar[r]^{c} & \xydot
    \ar[r]^{b} & \xydot \ar[r]^{\lambda a}& \xydot & {\cdots \, .} \ar[l]_{d} \\
  }
  \]

\end{example}

\begin{convention}
  From now on, unless explicitly needed for emphasis, we shall omit
  the degrees from unfolded diagrams. However, all homotopy letters in
  this article carry degrees, therefore, the reader should be aware
  that they are implicitly always present.
\end{convention}

\subsection{Infinite homotopy strings}\label{sec:infstrings}

These indecomposable complexes only occur when $\Lambda$ has infinite
global dimension: when the global dimension of $\Lambda$ is finite all
the complexes in $\sD$ are isomorphic to perfect complexes.  If
$\Lambda = \kk \Gamma/I$ has infinite global dimension, then $\Gamma$
contains oriented cycles with `full relations'.  Let $\mcC(\Lambda)$
denote the collection of arrows $a \in \Gamma_1$ such that there
exists a repetition-free cyclic path $a_n \cdots a_2 a_1$ in $\Gamma$
such that $a_{i+1} a_i \in I$ for $1 \leq i \leq n$ and $a_1 a_n \in
I$, where $a_1 = a$.
We need the following definition from \cite{Bo}.

\begin{definition} \label{def:antipath} A \emph{direct (resp. inverse)
    antipath} is a homotopy string where all homotopy letters are
  direct (resp. inverse) and arrows in the quiver.
\end{definition}

\begin{definition}
  Let $w = \prod_{k=n}^{1} (w_k,i_k,j_k)$ be a homotopy string. We say
  $w$ is
\begin{enumerate}[label=(\arabic*)]
\item \emph{left resolvable} if $(w_n,i_n,j_n)$ is direct (i.e. $j_n =
  i_n +1$), $i_n \leq i_k, j_k$ for all $1 \leq k < n$, and there
  exists $a \in \mcC(\Lambda)$ such that $(a,i_n - 1, i_n)w$ is a
  homotopy string. We shall say that $w$ is \emph{left resolvable by
    $a$}.
\item \emph{primitive left resolvable} if there is no direct antipath
  $\prod_{k=n}^t (w_k,i_k,j_k)$ such that $\prod_{k=t-1}^1
  (w_k,i_k,j_k)$ is left resolvable.
\end{enumerate}
\end{definition}

One can write down the obvious dual definitions of \emph{(primitive)
  right resolvable}. A homotopy string will be call \emph{(primitive)
  two-sided resolvable} if it is both (primitive) left resolvable and
(primitive) right resolvable.

If $w$ is left (resp. right) resolvable by $a$, then gentleness of
$\Lambda$ ensures that this is the unique arrow in $\Gamma_1$ by which
$w$ is left (resp. right) resolvable.  If $w$ is left resolvable then
$w^{-1}$ is right resolvable, and vice versa. We shall call a left or
right resolvable homotopy string that is not two-sided resolvable
\emph{one-sided resolvable}. There is then the obvious notion of
\emph{primitive one-sided resolvable}.

Suppose $w = \prod_{k=n}^1 (w_k, i_k, j_k)$ is left resolvable by $a_1
\in \mcC(\Lambda)$, which sits in the repetition-free cyclic path $a_m
\ldots a_1$ in $\Gamma$. We form the \emph{left infinite homotopy
  string} ${}^\infty w$ by concatenating infinitely many appropriately
shifted copies of the cycle on the left of $w$, i.e. the unfolded
diagram of ${}^\infty w$ has the form
\[
\xymatrix{
{}^\infty w \colon & \cdots \ar[r]^-{a_1} & \xydot \ar[r]^-{a_m} &
\xydot \ar@{.}[r] & \xydot \ar[r]^-{a_2} & \xydot \ar[r]^-{a_1} &
\xydot \ar[r]^-{w_n} & \xydot \ar@{-}[r]^-{w_{n-1}} & \xydot
\ar@{.}[r] & \xydot \ar@{-}[r]^-{w_1} & \xydot \,.
}
\]
Analogously, we form the \emph{right infinite homotopy string}
$w^\infty$ and the \emph{two-sided infinite homotopy strings}
${}^\infty w^\infty$ from a right resolvable homtopy string or
two-sided resolvable homotopy string, respectively.

If $w$ is one-sided left resolvable we obtain the corresponding
infinite string complex from the left infinite homotopy string
${}^\infty w$, or equivalently the right infinite homotopy string
$(w^{-1})^\infty$. Dually for $w$ right resolvable. For $w$ two-sided
resolvable we represent it by ${}^\infty w^\infty$ and
${}^\infty(w^{-1})^\infty$.

\subsection{The indecomposable objects of $\sD$}\label{sec:indecomp}

We first set up some notation.  A \emph{cyclic rotation} of the
homotopy band $w$ is a homotopy band of the form
\[
( w_k,i_k,j_k)( w_{k-1},i_{k-1},j_{k-1})\cdots( w_1,i_1,j_1)(
w_n,i_n,j_n)\cdots( w_{k+1}, i_{k+1},j_{k+1})\,.
\]
Consider the equivalence relation $\sim^{-1}$ generated by identifying
a homotopy string with its inverse, and the equivalence relation
$\sim^r$ generated by identifying a homotopy band with its cyclic
rotations and their inverses. The following will denote complete sets
of representatives of the specified objects under the given
equivalence relations:
\begin{align*}
\strings & \coloneqq \text{all homotopy strings under } \sim^{-1}; \\
\bands & \coloneqq \text{all homotopy bands under } \sim^r; \\
\stringsone & \coloneqq \text{all primitive one-sided resolvable homotopy strings under } \sim^{-1}; \\ 
\stringstwo & \coloneqq \text{all primitive two-sided resolvable homotopy strings under } \sim^{-1}.
\end{align*}

\begin{theorem}[{\cite[Theorem 3]{BM}}]
There are bijections 
\[
\ind{\sK} \bij \strings \sqcup (\bands \times \kk^* \times \bN)
\quad \text{and} \quad
\ind{\sD \setminus \sK} \bij \stringsone \cup \stringstwo.
\]
\end{theorem}

We mention here that a homotopy band can always be considered as a
homotopy string, thus we must take the disjoint union: a band gives
rise both to a string complex and a family of band complexes.

\section{Morphisms between indecomposable objects of
  $\sD$} \label{sec:morphisms}

In this section, we shall describe a canonical basis for the set of
homomorphisms between (finite or infinite) string and/or
one-dimensional band complexes; higher dimensional bands are dealt
with in Section \ref{sec:bands}.  The proof that the maps described
here do indeed form a basis is contained in Section \ref{sec:proof}.

We first set up some notation and define three canonical classes of
maps. Fix $w \in \strings \sqcup \bands$ and $\lambda_w \in
\kk^*$. Define
\[
Q_w \coloneqq
\left\{
\begin{array}{ll}
P_w            & \text{if $w \in \strings$;} \\
B_{w,\lambda_w,1} & \text{if $w \in \bands$.}
\end{array}
\right.
\]
We have abused notation here by allowing the scalar $\lambda_w$ to
disappear when dealing with $Q_w$ in the case that $w$ is a homotopy
band; it should be treated as implicitly present.

For $f \in \Hom_{\sC}(Q_v,Q_w)$ and a degree $t$, the map $f^t\colon
Q_v^t \rightarrow Q_w^t$ can be written as a matrix between the
finitely-many indecomposable summands of $Q_v^t$ and $Q_w^t$.  Each
entry of this matrix is a linear combination of paths (see Proposition
\ref{prop: paths}).  We refer to a single term in this sum as a
\emph{component} of $f^t$.  Moreover, a \emph{component of $f$} is
taken to be a component of $f^t$ for some degree $t$.  Throughout this
section we fix two homotopy strings or bands $v$ and $w$ and the
corresponding complexes $Q_v$ and $Q_w$.  We consider `maps' between
the unfolded diagrams of $Q_v$ and $Q_w$ which, at each projective,
will look like:
\[ 
\xymatrix@R=2mm{ 
Q_v \colon & \ar@{.}[r] & \xydot \ar@{-}[r]^{v_L} & \xydot \ar[d]^f \ar@{-}[r]^{v_R} & \xydot \ar@{.}[r] & \\
Q_w \colon & \ar@{.}[r] & \xydot \ar@{-}[r]_{w_L} & \xydot  \ar@{-}[r]_{w_R} & \xydot \ar@{.}[r] & \\
} 
\]
If $f$ is at the leftmost end of the unfolded string complex $Q_v$ we
say that $v_L$ is \emph{zero}; likewise for $v_R$, $w_L$, and $w_R$.
Again, $f$ is a linear combination of paths and we will use the term
\emph{component} to refer to a single summand of $f$.

In the next sections we shall define maps occurring in a canonical
basis of $\Hom_{\sC}(Q_v,Q_w)$.

\subsection{Single and double maps}\label{sec:single}

Suppose we are in the following situation: 
\begin{equation}
\xymatrix@R=2mm{ 
Q_v \colon & \ar@{.}[r] & \xydot \ar@{-}[r]^{v_L} & \xydot \ar[d]^f \ar@{-}[r]^{v_R} & \xydot \ar@{.}[r] & \\
Q_w \colon & \ar@{.}[r] & \xydot \ar@{-}[r]_{w_L} & \xydot  \ar@{-}[r]_{w_R} & \xydot \ar@{.}[r] & \\
} 
\label{eq:single}
\end{equation}
where $f$ is some non-stationary path in the quiver. 

\begin{definition}\label{def: single}
  A map in $\Hom_{\sC}(Q_v,Q_w)$ will be called a \emph{single map} if
  it has only one non-zero component whose unfolded diagram is as
  above and satisfies the following conditions:
\begin{itemize}
\item[\Lone] If $v_L$ is direct, then $v_Lf = 0$.
\item[\Ltwo] If $w_L$ is inverse, then $fw_L =0$.
\item[\Rone] If $v_R$ is inverse, then $v_Rf = 0$.
\item[\Rtwo] If $w_R$ is direct, then $fw_R = 0$.
\end{itemize}
Write $\single{v}{w}$ for the set of single maps $Q_v \to Q_w$. 
\end{definition}

\begin{example}\label{ex:single}
Again consider the Running Example.  Let 
\[
v = (dc,1,2)(b,2,3)(a,3,4) \text{ and  } w = (e,2,1)(f,1,0)(c,0,1)(b,1,2)(af,2,3).
\] 
Then we have the following: 
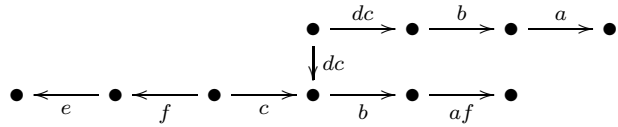
\[ 
\xymatrix@R=1mm{
& & & \xydot \ar[r]^{dc}\ar[dd]^-{dc}& \xydot \ar[r]^b& \xydot \ar[r]^a& \xydot\\
&&&&&& \\
\xydot & \xydot \ar[l]^e & \xydot \ar[l]^f \ar[r]_c & \xydot \ar[r]_b & \xydot \ar[r]_{af} & \xydot & \\
} 
\] 
which gives rise to a single map $P_v \to P_w$: 
\[
\xymatrix@R=6mm@C=7mm{ 
& P(1) \ar[r]^{dc} \ar[d]^-{\left(\begin{smallmatrix} dc \\0 \end{smallmatrix}\right)} & P(4) \ar[d]^{\left(\begin{smallmatrix} 0 \\0 \end{smallmatrix}\right)} \ar[r]^b & P(5) \ar[r]^a \ar[d]^-0 & P(3) \\
P(3) \ar[r]_-{\left(\begin{smallmatrix} c \\ f \end{smallmatrix}\right)} & P(4)\oplus P(2) \ar[r]_-{\left(\begin{smallmatrix} b & 0 \\0 & e \end{smallmatrix}\right)} & P(5) \oplus P(1) \ar[r]_-{\left(\begin{smallmatrix} af & 0 \end{smallmatrix}\right)} & P(2) &
} 
\]
\end{example}

Suppose we have the following situation:
\begin{equation}
\xymatrix@=1.2pc{ 
Q_v \colon & \ar@{.}[r] & \xydot \ar@{-}[r]^{v_L} & \xydot \ar[d]_{f_L} \ar[r]^{u_v} & \xydot \ar@{}[dl]|{(*)} \ar[d]^{f_R} \ar@{-}[r]^{v_R} & \xydot \ar@{.}[r] & \\
Q_w \colon & \ar@{.}[r] & \xydot \ar@{-}[r]_{w_L} & \xydot \ar[r]_-{u_w} & \xydot  \ar@{-}[r]_{w_R} & \xydot \ar@{.}[r] & \\
}
\label{eq:double}
\end{equation}
such that $(*)$ commutes for non-stationary paths $f_L,
f_R$.

\begin{definition}\label{def: double}
  If \Lone\ and \Ltwo\ hold for $f_L$ and \Rone\ and \Rtwo\ hold for
  $f_R$, then the diagram above induces a map $Q_v \to Q_w$ with two
  non-zero components in consecutive degrees given by $f_L$ and $f_R$.
  We call such a map a \emph{double map}; write $\double{v}{w}$ for
  the set of double maps $Q_v \to Q_w$.
\end{definition}

\begin{remark}\label{rem:alldouble} 
  To obtain all double maps $Q_v \to Q_w$ in terms of diagrams as
  above, it is necessary to consider overlaps between unfolded
  diagrams of $Q_v$ and any $Q_{w'}$ such that $Q_w \cong Q_{w'}$; see
  Section \ref{sec:indecomp} for the relevant equivalence relations on
  homotopy strings and bands.
\end{remark}

\begin{example} \label{ex:double}
Returning to the Running Example, consider the homotopy strings:
\[
v = (dc,-2,-1)(b,-1,0)(a,0,1) \text{ and  } w =
(af,3,2)(b,2,1)(c,1,0)(f,0,1)(e,1,2) \, .
\] 
The unfolded diagram
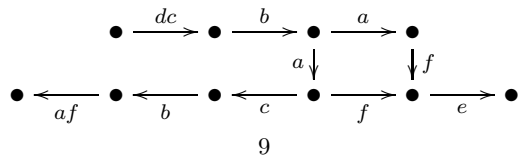
\[
\xymatrix@R=4mm{
& \xydot \ar[r]^{dc}& \xydot \ar[r]^b& \xydot \ar[d]_{a} \ar[r]^a& \xydot \ar[d]^{f} & \\
\xydot & \ar[l]^{af} \xydot & \xydot \ar[l]^b & \xydot \ar[l]^c \ar[r]_f & \xydot \ar[r]_e & \xydot  
}
\] 
gives rise to the double map:
\[ 
\xymatrix@R=6mm{  P(1) \ar[r]^{dc} & P(4) \ar[r]^b & P(5) \ar[r]^a \ar[d]_-a & P(3)\ar[d]^-{\left(\begin{smallmatrix} 0 \\ f \end{smallmatrix}\right)} & & \\
& & P(3) \ar[r]_-{\left(\begin{smallmatrix} c \\ f \end{smallmatrix}\right)} & P(4)\oplus P(2) \ar[r]_-{\left(\begin{smallmatrix} b & 0 \\0 & e \end{smallmatrix}\right)} & P(5) \oplus P(1) \ar[r]_-{\left(\begin{smallmatrix} af & 0 \end{smallmatrix}\right)} & P(2)
} 
\]
\end{example}

\begin{notation} \label{not:maps} Let $f \in \single{v}{w}$ and
  suppose the unique non-zero component of $f$ corresponds to the path
  $p \colon x \path y$ in $\Gamma$. Then we write $f=(p)$. Similarly,
  for $f \in \double{v}{w}$ whose two non-zero components correspond
  to the paths $p \colon x \path y$ and $q \colon x' \path y'$, we
  write $f=(p,q)$. The notation should be suggestive of an infinite
  vector, in which all entries are zero apart from those which are
  written.
\end{notation}

We next highlight two important classes of single and double maps.

\begin{definition}\label{def:singleton}
  Recall the setup in \eqref{eq:single}. A map $f \in \single{v}{w}$
  will be called a \emph{singleton single map} when its non-zero
  component is given by a path $p$ and we are in the situation of the
  following unfolded diagrams (up to inverting one of the homotopy
  strings):
 
 \begin{align*}
(i) \quad 
\xymatrix@R=5mm{
\ar@{.}[r] & \xydot \arr^-{v_r} & \xydot \ar[d]^-{p} &                   & \\
           &                    & \xydot \arr_-{w_s} & \xydot \ar@{.}[r] &  
}
& \quad & (ii) \quad
\xymatrix@R=5mm{
\ar@{.}[r] & \xydot \arr^-{v_{i+1}} & \xydot \ar[d]^-{p} \ar[r]^-{v_i}_-{= p p_R} & \xydot \ar@{.}[r] & \\
           &                    & \xydot \arr_-{w_s}                         & \xydot \ar@{.}[r] &  
} \\
(iii) \quad 
\xymatrix@R=5mm{
\ar@{.}[r] & \xydot \arr^-{v_r}                 & \xydot \ar[d]^-{p} &                   & \\
\ar@{.}[r] & \xydot \ar[r]^-{w_{j}}_-{=p_L p}  & \xydot \arr_-{w_{j-1}} & \xydot \ar@{.}[r] &  
}
& \quad & (iv) \quad
\xymatrix@R=5mm{
\ar@{.}[r] & \xydot \arr^-{v_{i+1}}                 & \xydot \ar[d]^-{p} \ar[r]^-{v_i}_-{= p p_R} & \xydot \ar@{.}[r] & \\
\ar@{.}[r] & \xydot \ar[r]^-{w_j}_-{=p_L p}  & \xydot \arr_-{w_{j+1}} & \xydot \ar@{.}[r] & \, ,  
}
\end{align*}
where $r \in \{1,n\}$, $s \in \{1,m\}$, either $v_r$ is inverse (or
zero) or $v_r p =0$, and either $w_s$ is inverse (or zero) or $p w_s =
0$.  We denote the set of all singleton single maps $Q_v \to Q_w$ by
$\singleone{v}{w}$.
\end{definition}

\begin{definition}\label{def:singleton-double}
  Recall the setup in \eqref{eq:double}. A map $f=(f_L, f_R) \in
  \double{v}{w}$ will be called a \emph{singleton double map} if it
  satisfies the following condition (or its dual).
\begin{itemize}
\item[\D] There exists a non-stationary path $f'$ such that $v_0 = f_L
  f'$ and $w_0 = f'f_R$.
\end{itemize}
We denote the set of all singleton double maps $Q_v \to Q_w$ by
$\doubleone{v}{w}$.
\end{definition}

Condition \D\ means that to find singleton double maps, it is
sufficient to look for homotopy letters $v_i$ of $v$ and $w_j$ of $w$
such that $v_i$ and $w_j$ sit in the same degrees, with the same
orientation, and a `proper right substring' of $v_i$ is a `proper left
substring' of $w_j$.

\subsection{Graph maps and quasi-graph maps} \label{sec:graph}

Suppose the unfolded diagrams of $Q_v$ and $Q_w$ overlap as follows:
\begin{equation*} 
\xymatrix@R=2mm{ 
\text{degrees:} &             &                                         & t_p                                                & t_{p-1}                              &                         & t_1                                & t_0                                &                                    & \\
Q_v  \colon     &  \ar@{.}[r] & \xydot \ar@{.>}[d]_{f_L} \ar@{.}[r]^{v_L} & \xydot \ar@{}[dl]|{(*)} \ar@{=}[d] \ar@{-}[r]^{u_p} & \xydot \ar@{=}[d] \ar@{-}[r]^{u_{p-1}} & \cdots \ar@{-}[r]^{u_2}  & \xydot \ar@{=}[d] \ar@{-}[r]^{u_1} & \xydot \ar@{=}[d] \ar@{.}[r]^{v_R} & \xydot \ar@{}[dl]|{(**)} \ar@{.>}[d]^{f_R} \ar@{.}[r] &  \\
Q_w \colon      &  \ar@{.}[r] & \xydot \ar@{.}[r]_{w_L}                  & \xydot \ar@{-}[r]_{u_p}                             & \xydot \ar@{-}[r]_{u_{p-1}}            & \cdots \ar@{-}[r]_{u_2}  & \xydot \ar@{-}[r]_{u_1}            & \xydot \ar@{.}[r]_{w_R}             & \xydot \ar@{.}[r]                  &  \\
\text{degrees:} &             &                                         & t_p                                                & t_{p-1}                               &                         & t_1                               & t_0                                &                                    &
}
\end{equation*}
such that $v_L \neq w_L$ and $v_R \neq w_R$.  The double lines
represent isomorphisms and all of the squares with solid lines commute
as paths in the quiver.  Consider the following \emph{left endpoint
  conditions}.
\begin{itemize}
\item[\LGone] The arrows $v_L$ and $w_L$ are either both direct or
  both inverse and there exists some (scalar multiple of a)
  non-stationary path $f_L$ such that the square $(*)$ commutes.
\item[\LGtwo] The arrows $v_L$ and $w_L$ are neither both direct nor
  both inverse.  In this case, if $v_L$ is non-zero then it is inverse
  and if $w_L$ is non-zero then it is direct.
\end{itemize}
There are dual \emph{right endpoint conditions}, which we call \RGone\
and \RGtwo, respectively.

\begin{definition}\label{def:graphmaps}
  If one of \LGone\ or \LGtwo\ hold and one of \RGone\ or \RGtwo\
  hold, the diagram induces a map $Q_v \to Q_w$, whose non-zero
  components are exactly those described by the diagram. Such maps are
  called \emph{graph maps}; write $\graph{v}{w}$ for the set of graph
  maps $Q_v \to Q_w$.
\end{definition}

\begin{example}
Consider the algebra in the Running Example.  Let 
\[
v = (dc,0,1)(b,1,2)(a,2,3) \hbox{ and  } w = (e,2,1)(f,1,0)(c,0,1)(b,1,2)(af,2,3).
\]  
The unfolded diagram on the left gives rise to the graph map on the
right.
\[ 
\xymatrix@R=1.5mm{
& & \xydot \ar[dd]^d \ar[r]^{dc}& \xydot \ar@{=}[dd] \ar[r]^b& \xydot \ar@{=}[dd] \ar[r]^a& \xydot \ar[dd]^f \\
&&&&& \\
\xydot & \xydot \ar[l]^e & \xydot \ar[l]^f \ar[r]_c & \xydot \ar[r]_b & \xydot \ar[r]_{af} & \xydot }
\quad
\xymatrix@R=5mm@C=7mm{  \smxy{P(1)} \ar[d]^d \ar[r]^{dc} & \smxy{P(4)} \ar[d]^{\left( \begin{smallmatrix} 1 \\ 0 \end{smallmatrix}\right) } 
\ar[r]^b & \smxy{P(5)} \ar[d]^{\left(\begin{smallmatrix} 1 \\ 0 \end{smallmatrix}\right)} \ar[r]^a  & \smxy{P(3)} \ar[d]^f \\
\smxy{P(3)} \ar[r]_-{\left(\begin{smallmatrix} c \\ f \end{smallmatrix}\right)} & \smxy{P(4)\oplus P(2)} \ar[r]_-{\left(\begin{smallmatrix} b & 0 \\0 & e \end{smallmatrix}\right)} & \smxy{P(5) \oplus P(1)} \ar[r]_-{\left(\begin{smallmatrix} af & 0 \end{smallmatrix}\right)} & \smxy{P(2)}
}
\]
\end{example}

\begin{definition}\label{def:quasigraph}
  If none of the conditions \LGone, \LGtwo, \RGone\ nor \RGtwo\ hold,
  then the diagram no longer induces a map, however we shall say that
  there is a \emph{quasi-graph map} from $Q_v$ to $Q_w$.
\end{definition}

\begin{definition}
  Consider the diagram representing a quasi-graph map $Q_v \to
  \Sigma^{-1}Q_w$:

 \[ \xymatrix@R=2mm{ Q_v  \colon     &  \ar@{.}[r] & \xydot  \ar@{.}[r]^{v_L} & \xydot  \ar@{=}[d] \ar@{-}[r]^{u_p} & \xydot \ar@{=}[d] \ar@{-}[r]^{u_{p-1}} & \cdots \ar@{-}[r]^{u_2}  & \xydot \ar@{=}[d] \ar@{-}[r]^{u_1} & \xydot \ar@{=}[d] \ar@{.}[r]^{v_R} & \xydot  \ar@{.}[r] &  \\
\Sigma^{-1} Q_w \colon      &  \ar@{.}[r] & \xydot \ar@{.}[r]_{w_L}                  & \xydot \ar@{-}[r]_{u_p}                             & \xydot \ar@{-}[r]_{u_{p-1}}            & \cdots \ar@{-}[r]_{u_2}  & \xydot \ar@{-}[r]_{u_1}            & \xydot \ar@{.}[r]_{w_R}             & \xydot \ar@{.}[r]                  &  \\
}
\] 

\noindent Then there exist $p$ single maps $Q_v \to Q_w$ given by the
paths $u_p, u_{p-1}, ... , u_1$ in the appropriate degrees.  There are
also two (single or double) maps with non-zero components given by
$v_L$, $w_L$, $v_R$ or $w_R$ in the appropriate degree.  For example,
if $v_L$ is direct and $w_L$ is inverse than there is a double map
$(v_L,w_L)$.  We call these maps the associated \emph{quasi-graph map
  representatives}.  Let $\quasi{v}{w}$ be a fixed set of quasi-graph
map representatives $Q_v \to Q_w$, one for each quasi-graph map from
$Q_v$ to $\Sigma^{-1} Q_w$.
\end{definition}

\begin{example}\label{ex: homotopy class}
Let us return to the Running Example. Consider the homotopy strings 
\begin{align*}
v & = (f,0,1)(e,1,2)(dc,2,3)(b,3,4)(a,4,5)(d,5,4), \ \text{and} \\ 
w & = (c,0,-1)(f,-1,0)(e,0,1)(dc,1,2)(b,2,3)(a,3,4).
\end{align*}
The following diagram describes a quasi-graph map $Q_v \to \Sigma^{-1} Q_w$:
\[ 
\xymatrix@R=2mm{ 
Q_v  \colon     &    & \xydot \ar@{=}[d] \ar[r]^-{f} & \xydot \ar@{=}[d] \ar[r]^{e} & \xydot \ar@{=}[d] \ar[r]^{dc} & \xydot \ar[r]^{b} \ar@{=}[d] & \xydot \ar@{=}[d] \ar[r]^{a} & \xydot \ar@{=}[d] \ar@{<-}[r]^{d} & \xydot \,  \\
\Sigma^{-1}Q_w \colon      &  \xydot \ar@{<-}[r]_-{c} & \xydot \ar[r]_-{f}  & \xydot \ar[r]_{e}                             & \xydot \ar[r]_{dc}            & \xydot \ar[r]_{b}  & \xydot \ar[r]_{a}            & \xydot         &  
}
\] 
and the associated quasi-graph map representatives $Q_v \to Q_w$ are 
\[ 
\{ (-c),(f),(-e),(dc),(-b),(a),(-d) \}. 
\]
\end{example}

\begin{remarks}
We highlight the following.
\begin{enumerate}
\item \label{rem:infinitegraphmaps} (Quasi-)graph maps extend to
  infinite homotopy strings in the obvious way; if a left
  (respectively right) endpoint condition holds and the substrings to
  the right (respectively left) are equal and infinite then we can
  define a (quasi-)graph map with infinitely many components which are
  isomorphisms. Cf. \cite{C-B2} for infinite graph maps between
  modules.

\item To standardise what we mean by a (quasi-)graph map we take the
  following conventions. If $v$ and $w$ are both homotopy strings,
  then each isomorphism is an identity.  If one of $v$ and $w$ is a
  homotopy band, then the leftmost isomorphism in the diagram is an
  identity and the remaining isomorphisms will be determined by this.
  Note that any other choice of isomorphisms will result in a scalar
  multiple of such a standardised map.

\item It is necessary to replace $w$ by an equivalent homotopy
  string/band to obtain all of the possible (quasi)-graph maps $Q_v
  \to Q_w$; cf. Remark \ref{rem:alldouble}.

\item \label{item:pathology} There is one pathological example arising
  from the identity map $B_{w,\lambda,1} \to B_{w, \lambda,1}$. This
  is a graph map since it `travels' the whole way around the band. In
  particular, there are no real endpoint conditions. As such it also
  defines a quasi-graph map $B_{w,\lambda,1} \to \Sigma^{-1}
  B_{w,\lambda,1}$. This situation will be treated in more detail in
  Section~\ref{sec:bands}.
\end{enumerate}
\end{remarks}

\subsection{The main theorem}

We have now assembled all the maps that occur in a canonical basis of
$\Hom_{\sD}(Q_v,Q_w)$ and can state our main result succinctly as the
following.

\begin{theorem} \label{thm:main} The set $\basisD{v}{w} \coloneqq
  \singleone{v}{w} \cup \doubleone{v}{w} \cup \graph{v}{w} \cup
  \quasi{v}{w}$ is a $\kk$-linear basis for $\Hom_\sD(Q_v,Q_w)$.
\end{theorem}

The next section concerns the proof of this result.

\begin{example} Consider the following homotopy strings over the
  Running Example:
\begin{align*}
  v &= (a,-1,0)(c,0,1)(b,1,2)\\
  w &= (e,2,3)(d,3,4)(a,4,3)(b,3,2)(cd,2,1)(e,1,0)(f,0,-1)(c,-1,0)(b,0,1)(a,1,2)
\end{align*}
The set $\basisD{v}{w}$ has two elements.  The first is a singleton
single map $(af)$ in degree $2$ from the right end-point of $v$ to the
left end-point of $w$.  The second is given by the
following quasi-graph map $Q_v \rightarrow \Sigma^{-1}Q_w$, where the
associated quasi-graph map representatives $Q_v \rightarrow Q_w$ are
indicated by dashed arrows (single maps) and dotted arrows (double map):
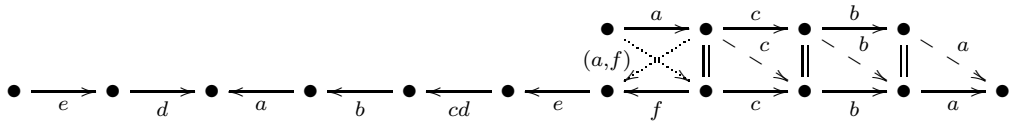
\[
\xymatrix@R=4mm{
  & & & & & & \xydot \ar[r]^-a \ar@{..>}[dr] \ar@{}[d]|{(a,f)} & \xydot \ar@{=}[d]
  \ar@{..>}[dl] \ar[r]^-c
  \ar@{-->}[dr]^-c & \xydot \ar@{=}[d] \ar[r]^-b \ar@{-->}[dr]^-b&
  \xydot \ar@{=}[d] \ar@{-->}[dr]^-a& \\
  \xydot \ar[r]_-e & \xydot \ar[r]_-d & \xydot & \xydot \ar[l]^-a
  & \xydot \ar[l]^-b & \xydot \ar[l]^-{cd} & \xydot \ar[l]^-e
  & \xydot \ar[l]^-f \ar[r]_-c & \xydot \ar[r]_-b & \xydot \ar[r]_-a & \xydot
}
\]
\end{example}

\section{Proof of the main theorem} \label{sec:proof}

Let $v,w \in \strings \sqcup \bands$ and $Q_v$, $Q_w$ be as in the
previous section.  We shall split up the proof of
Theorem~\ref{thm:main} into two parts. The first part establishes a
canonical basis for $\Hom_{\sC}(Q_v,Q_w)$. In the second part, we
identify which elements of this basis are homotopic or null-homotopic.

\subsection{A basis at the level of complexes}

In this section, we establish the following:

\begin{proposition}\label{prop:basis}
  The set $\basisC{v}{w} \coloneqq \graph{v}{w} \cup \single{v}{w}
  \cup \double{v}{w}$ is a $\kk$-linear basis for
  $\Hom_{\sC}(Q_v,Q_w)$.
\end{proposition}

The proof of Proposition~\ref{prop:basis} is inspired by \cite[Section
1.4]{C-B2}. We first need two technical lemmas from which we will
deduce that $\basisC{v}{w}$ is a linearly independent set.

\begin{lemma}\label{lemma: zero squares}
Suppose we have the following situation:
\[ 
\xymatrix@R=5mm{
\ar@{.}[r] & \xydot \ar[d]_{f_L} \ar@{-}[r]^{v_L} & \xydot \ar@{}[dl]|{(*)} \ar[d]_{f_C} \ar@{-}[r]^{v_R} & \xydot \ar@{}[dl]|{(**)} \ar[d]^{f_R} \ar@{.}[r]& \\
\ar@{.}[r] & \xydot \ar@{-}[r]_{w_L} & \xydot \ar@{-}[r]_{w_R} & \xydot \ar@{.}[r]&
}
\] 
where $f_C$ is a non-stationary path and $v_L$ is direct if and only
$w_L$ is direct; similarly for $v_R$ and $w_R$.  If the squares $(*)$
and $(**)$ commute, then at most one of $(*)$ and $(**)$ has a
non-zero commutativity relation.
\end{lemma}

\begin{proof}
  We analyse the case when all four homotopy letters $v_L$, $v_R$,
  $w_L$ and $w_R$ are direct; the remaining cases are analogous.

  Suppose $(*)$ has a non-zero commutativity relation.  Then $v_L f_C
  = f_Lw_L \neq 0$.  It follows that the paths $f_C$ and $w_L$ start
  with in the same arrow.  By the definitions of homotopy strings and
  bands, $w_Lw_R = 0$ and, by condition \ref{item:length2} of
  gentleness, we must also have that $f_Cw_R=0$.  That is,
  $f_Cw_R=v_Rf_R=0$ and $(**)$ has a zero commutativity relation.

  Dually, whenever $(**)$ has a non-zero commutativity relation, $(*)$
  has a zero commutativity relation.
\end{proof}

\begin{lemma}\label{lemma: basis maps determined}
  The unfolded diagram giving rise to a graph, single or double map is
  completely determined by any non-zero component.
\end{lemma}

\begin{proof}
  We give a proof for the case where we have a non-zero component of a
  graph map.  The arguments for single and double maps are analogous.

  We make use of the diagram from Lemma \ref{lemma: zero squares}.
  Suppose that $f_C$ is a non-zero component of a graph map.  We show
  that, if it exists, $f_L$ is completely determined. Note that if
  $v_L$ and $w_L$ have different orientations then the arrow $f_L$
  does not make sense since the degrees will not match.  So suppose
  $v_L$ and $w_L$ are either both direct or both inverse.

  Suppose $f_C$ is an isomorphism and $v_L$ and $w_L$ are both direct.
  Then the square $(*)$ must commute so $f_L$ is the unique path such
  that $f_Lw_L = v_Lf_C$ (if the path is stationary then $f_L$ is an
  isomorphism).  Similarly, when $v_L$ and $w_L$ are both inverse then
  $f_L$ is the unique path such that $v_Lf_L = f_Cw_L$.

  Suppose $f_C$ is given by a path, then according to Lemma
  \ref{lemma: zero squares} we must be at the left or right endpoint
  of the diagram.  If $v_L$ is direct and $v_Lf_C = 0$, then we must
  be at the left end of the diagram and $f_L=0$.  If $v_Lf_C\neq0$,
  then we are are at the right end of the diagram and $f_L$ is an
  isomorphism.  Similarly, when $w_L$ is inverse and $f_Cw_L = 0$,
  then $f_L=0$.  If $f_Cw_L \neq 0$, then $f_L$ is an isomorphism.

  We can apply dual arguments to conclude that the diagam is also
  completely determined to the right (i.e. $f_C$ determines $f_R$).
\end{proof}

\begin{corollary}
  The set $\basisC{v}{w} = \graph{v}{w} \cup \single{v}{w} \cup
  \double{v}{w}$ is linearly independent.
\end{corollary}

\begin{proof}
  Let $0 \neq b \in \basisC{v}{w}$. We show that $b$ cannot be written
  as a linear combination of other elements of $\basisC{v}{w}$.
  Suppose $b = \sum_{i=1}^n k_ib_i$ for pairwise different
  $b_1,...,b_n \in \basisC{v}{w}$ and $k_i \in \kk$.  Let $f_C$ be a
  non-zero component of $b$.  Then some $b_i$, $1 \leq i \leq n$, must
  also have this non-zero component because the algebra is gentle and
  so only has zero relations. By Lemma \ref{lemma: basis maps
    determined}, $b=b_i$ and so $k_i=1$ and $\sum_{j=1}^{i-1}k_jb_i +
  \sum_{j=i+1}^nk_jb_j = 0$ as required.
\end{proof}

\begin{proof}[Proof of Proposition \ref{prop:basis}]
  It suffices to show that $\basisC{v}{w}$ spans
  $\Hom_{\sC}(Q_v,Q_w)$.
  Let $0 \neq h \in \Hom_{\sC}(Q_v,Q_w)$.  Then there is some degree
  $t$ such that $h^t \colon Q_v^t \to Q_w^t$ has a non-zero component
  $h^t_{ab}\colon P(\phi_v(a)) \to P(\phi_w(b))$.  By the shape of
  homotopy strings and bands, $P(\phi_v(a))$ and $P(\phi_w(b))$ are
  each connected to at most two non-zero components of the
  differential.  We must therefore consider the unfolded diagrams (as
  in Lemma \ref{lemma: zero squares} but with $h_{ab}^t$ in place of
  $f_C$). Without loss of generality, assume that $h_{ab}^t$ is an
  isomorphism or a scalar multiple of a path.  By Lemma \ref{lemma:
    basis maps determined}, there is a unique (scalar multiple of an)
  element of $\basisC{v}{w}$ with this component and, by Lemma
  \ref{lemma: zero squares}, this must be a summand of $h$.  If this
  is not the whole of $h$, then we choose another non-zero component
  of $h$ and continue until we have found a complete decomposition of
  $h$.  Thus $h \in \Span \basisC{v}{w}$.
\end{proof}

Proposition~\ref{prop:basis} gives us canonical bases for the Hom
spaces between indecomposable complexes of $\sD$ considered as objects
of $\sC$. We next turn our attention to homotopy classes of these
maps. The following section highlights the strategy of our approach.

\subsection{The strategy for constructing homotopy
  classes}\label{sec:homotopy-general}

We first recall the general definition; we direct the reader to a
standard textbook on homological algebra, for example
\cite{Hilton-Stammbach, Weibel} for more information regarding
homotopies.

\begin{definition}
  Let $(\cxP, d^{\bullet})$ and $(\cxQ, \dd^{\bullet})$ be complexes
  in $\sC$.  Then maps $f,g \colon \cxP \to \cxQ$ are said to be
  \emph{homotopic}, written $f \simeq g$, if there are maps $h^i
  \colon P^i \to Q^{i-1}$ such that $f^i - g^i = d^i h^{i+1} + h^i
  \dd^{i-1}$ (where, as before, the maps are composed from left to
  right, see Remark \ref{Rmk: transpose}).  The family of maps
  $\{h^i\}$ is called a \emph{homotopy} from $f$ to $g$. If $g=0$ then
  $f$ is called \emph{null-homotopic}.
\end{definition}

Consider a map $f \in \basisC{v}{w}$ and let $p$ be a component of
$f$. We can write down an unfolded diagram representation of this
component as follows.
\[
\xymatrix@R=5mm{
Q_v \colon & \ar@{.}[r]  & \xydot \arr^-{v_L}  & \xydot \ar[d]^-{p} \arr^-{v_R}   & \xydot \ar@{.}[r]  &  \\
Q_w \colon & \ar@{.}[r]  & \xydot \arr_-{w_L} & \xydot \arr_-{w_R}               & \xydot \ar@{.}[r] & }
\]
Thus, the corresponding component of a homotopic map must be given by
\[
q = p + \big( \alpha v_L a + \beta v_R b + \gamma c w_L + \delta d w_R  \big), 
\]
for some scalars $\alpha, \beta, \gamma, \delta \in \kk$ and paths
$a,b,c,d$ in the quiver. If the composition of paths does not make
sense we take the corresponding scalar to be zero. For instance, if
$v_L$ is direct then $\alpha =0$.  Therefore, in order to construct
homotopies between maps in $\basisC{v}{w}$ it is enough to look at
ways to construct the path $p$ by `completing' differential
components.

\begin{definition} \label{def:homotopy-class} We denote by
  $\homotopy{f}$ the set of maps $f'$ such that $f \simeq f'$ and
  $f'=\lambda g$ for some $g \in \basisC{v}{w}$ and $\lambda \in
  \kk^*$.
\end{definition}

\begin{remark}\label{rem:homotopy-classes}
  The set $\homotopy{f}$ is not the same as the homotopy class of $f$,
  however, if $g \simeq f$ then the decomposition of $g$ into a linear
  combination of elements of $\basisC{v}{w}$ will consist of elements
  of $\homotopy{f} \cup \homotopy{0}$ only.  Hence, it suffices to
  determine the sets $\homotopy{f}$.  If $f$ is non-zero and
  $\homotopy{f}$ is a singleton set then we will say that $f$ belongs
  to a \emph{singleton homotopy class}.
\end{remark}

\subsection{Basic maps $f$ such that $\homotopy{f}$ is not
  singleton} \label{sec:not-singleton}

Here we start with $f \in \single{v}{w} \cup \double{v}{w}$; we shall
see in Section \ref{sec:singletons} that we do not need to consider $f
\in \graph{v}{w}$.  The reader may find it helpful to recall
Definitions~\ref{def:graphmaps} and \ref{def:quasigraph} and the
corresponding endpoint conditions.

\begin{proposition} \label{prop:homotopy-classes} Suppose $f \in
  \single{v}{w} \cup \double{v}{w}$ such that $\homotopy{f} \neq
  \{f\}$. Then $f$ is not null-homotopic if and only if the elements
  of $\homotopy{f}$ are in one-to-one correspondence with the
  representatives of a quasi-graph map $Q_v \to \Sigma^{-1} Q_w$.
  Moreover,
  \begin{enumerate}
  \item \label{item:left-double-end} $\homotopy{f}$ has a double map
    on the left if $v_L \neq 0$ is direct and $w_L\neq 0$ is inverse.
  \item \label{item:right-double-end} $\homotopy{f}$ has a double map
    on the right if $v_R \neq 0$ is inverse and $w_R \neq 0$ is
    direct.
  \item \label{item:single-end} Otherwise, $\homotopy{f}$ ends with a
    single map on the left or right.
  \end{enumerate}
\end{proposition}

\begin{proof}
  For simplicity, we consider only the case $f \in \single{v}{w}$. The
  case $f \in \double{v}{w}$ is similar. The setup is the following,
  where by abuse of notation we have denoted the map and its unique
  non-zero component by $f$:
\[
\xymatrix@R=4mm{
Q_v \colon \ar[d]_f & \ar@{.}[r]  & \xydot \arr^-{v_{-2}} & \xydot \arr^-{v_{-1}}  & \xydot \ar[d]^-{f} \arr^-{v_0}   & \xydot \arr^-{v_1} & \xydot \ar@{.}[r]  &  \\
Q_w \colon & \ar@{.}[r]  & \xydot \arr_-{w_{-2}} & \xydot \arr_-{w_{-1}} & \xydot \arr_-{w_0}               & \xydot \arr_-{w_1} & \xydot \ar@{.}[r] & }
\]
As explained in Section \ref{sec:homotopy-general}, the components
corresponding to $f$ in any homotopic map can be constructed only from
four possible paths $v_{-1}'$,$v_0'$, $w_{-1}'$ or $w_0'$ illustrated
in the following diagrams:
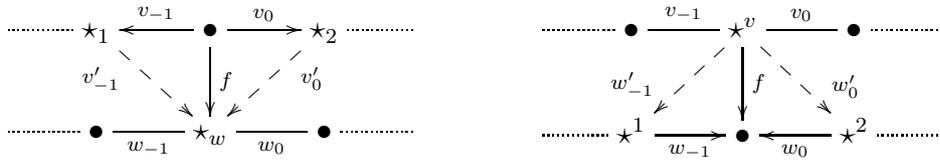
\[
 \quad
\xymatrix{
\ar@{.}[r]  & \star_1 \ar@{<-}[r]^-{v_{-1}} \ar@{-->}[dr]_<<{v'_{-1}}  & \xydot \ar[d]^-{f} \ar[r]^-{v_0}   & \star_2 \ar@{.}[r] \ar@{-->}[dl]^<<{v'_0} &  \\
\ar@{.}[r]  & \xydot \arr_-{w_{-1}} & \star_w \arr_-{w_0}               & \xydot \ar@{.}[r] & }
\quad  \quad  \quad
\xymatrix{
\ar@{.}[r]  & \xydot \arr^-{v_{-1}}  & \star^v \ar[d]^-{f} \arr^-{v_0} \ar@{-->}[dl]_>>{w'_{-1}} \ar@{-->}[dr]^>>{w'_0}  & \xydot \ar@{.}[r]  &  \\
\ar@{.}[r]  & \star^1 \ar[r]_-{w_{-1}}  & \xydot                 & \star^2 \ar[l]^-{w_0}\ar@{.}[r] & }
\]
Observe that $f$ can be immediately seen to be null-homotopic in the
following cases:
\begin{itemize}
\item[\None] If any of $v'_{-1}$, $v'_0$, $w'_{-1}$ and $w'_0$ (exist
  and) are non-stationary paths making one of the triangles commute.
\item[\Ntwo] If there are no arrows out of $\star_w$ and either $f=
  v_{-1}$ and there is no other arrow in $v$ into $\star_1$, or,
  $f=v_0$ and there is no other arrow in $v$ into $\star_2$.
\item[\Nthree] If there are no arrows into $\star^v$ and either
  $f=w_{-1}$ and there is no other arrow in $w$ out of $\star^1$, or
  $f=w_0$ and there is no other arrow in $w$ out of $\star^2$.
\end{itemize}
The conditions \None--\Nthree correspond to the endpoint conditions
\LGone, \LGtwo, \RGone\ and \RGtwo\ of Definition~\ref{def:graphmaps}
for graph maps.

Suppose we are not in any of the cases \None--\Nthree. Since
$\homotopy{f} \neq \{f\}$, at least one, but possibly both, of the
following must hold:
\begin{itemize}
\item $f$ is \emph{built from the source differential}, i.e.\ $f =
  v_{-1}$ or $f = v_0$ (but not both); or
\item $f$ is \emph{built from the target differential}, i.e.\ $f =
  w_{-1}$ or $f = w_0$ (but not both).
\end{itemize}

Suppose $f$ can be built from the source differential; the argument
when $f$ can be built from the target differential is dual. Without
loss of generality, we assume $f = v_0$; if $f=v_{-1}$, invert the
homotopy string $v$ and re-label $v_{-1}$ as $v_0$. The differential
$v_0$ will be called a \emph{used differential}, because it has
already been used to construct one of the single or double maps,
namely $f$ in this instance.  There are two cases.

\Casecolon{There are no arrows out of $\star_w$} We must be in the
situation that $w_{-1}$ is direct, $w_0$ is inverse, and $v_1$ exists
and is inverse -- for otherwise we would be in case \Ntwo\ above. The
following diagram describes the situation.
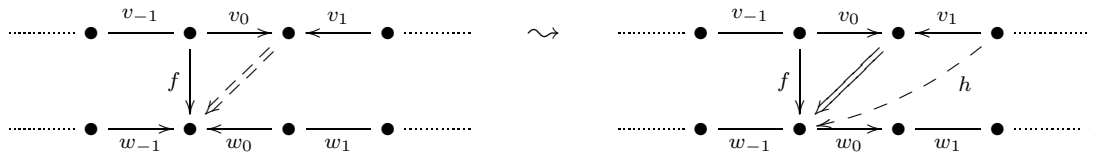
\[
\xymatrix{
\ar@{.}[r]  & \xydot \arr^-{v_{-1}}  & \xydot \ar[d]_-{f} \ar[r]^-{v_0}   & \xydot \ar@{<-}[r]^-{v_1}  \ar@{==>}[dl] & \xydot \ar@{.}[r] &  \\
\ar@{.}[r]  & \xydot \ar[r]_-{w_{-1}} & \xydot \ar@{<-}[r]_-{w_0}               & \xydot \arr_-{w_1} & \xydot \ar@{.}[r] & }
\quad \path \quad
\xymatrix{
\ar@{.}[r]  & \xydot \arr^-{v_{-1}}  & \xydot \ar[d]_-{f} \ar[r]^-{v_0}   & \xydot\ar@{=>}[dl] & \xydot \ar[l]_-{v_1} \ar@{.}[r] \ar@/^0.5pc/@{-->}[dll]^<<<{h} &  \\
\ar@{.}[r]  & \xydot \arr_-{w_{-1}} & \xydot \ar[r]_-{w_0}               & \xydot \arr_-{w_1} & \xydot \ar@{.}[r] & \, . } 
\]
In particular, $f = (v_0)$ is homotopic to the single map $-h =
-(v_1)$. We can now see part of the quasi-graph map constructed:
namely, the equality written diagonally. We now wish to continue by
building $h$ from a differential. The differential $v_1$ has already
been used, so in order to continue, we must see whether $h = w_0$ or
$h = w_{-1}$, and then use the dual argument for maps built from the
target differential.

\Casecolon{There is an arrow out of $\star_w$} Without loss of
generality, assume $w_0 \neq 0$ is direct. Note that only one of $w_0$
or $w_{-1}$ may be an arrow out of $\star_w$ -- otherwise $f$ cannot
be a well-defined single map. Then there exists
\begin{itemize}
\item a single map $g = w_0 \colon P(s(v_0)) \to P(s(w_0))$ if $v_1$
  is direct or zero; or
\item a double map $(h,g) = (v_1 , w_0)$, whenever $v_1$ is inverse:
\end{itemize}
The used differentials are $v_0$ and $w_0$ in the first case; in the
second case $v_1$ is additionally a used differential. The situation
is illustrated below:
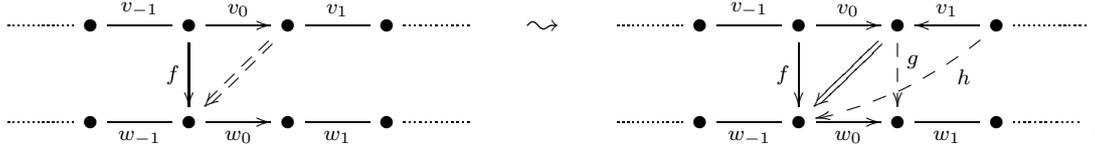
\[
\xymatrix{
\ar@{.}[r]  & \xydot \arr^-{v_{-1}}  & \xydot \ar[d]_-{f} \ar[r]^-{v_0}   & \xydot \arr^-{v_1}  \ar@{==>}[dl] & \xydot \ar@{.}[r] &  \\
\ar@{.}[r]  & \xydot \arr_-{w_{-1}} & \xydot \ar[r]_-{w_0}               & \xydot \arr_-{w_1} & \xydot \ar@{.}[r] & }
\quad \path \quad
\xymatrix{
\ar@{.}[r]  & \xydot \arr^-{v_{-1}}  & \xydot \ar[d]_-{f} \ar[r]^-{v_0}   & \xydot\ar@{=>}[dl] \ar@{-->}[d]^<<<{g} & \xydot \ar[l]_-{v_1} \ar@{.}[r] \ar@/^0.5pc/@{-->}[dll]^<<<{h} &  \\
\ar@{.}[r]  & \xydot \arr_-{w_{-1}} & \xydot \ar[r]_-{w_0}               & \xydot \arr_-{w_1} & \xydot \ar@{.}[r] & \, . } 
\]
This gives $-g \in \homotopy{f}$ or $-(h,g) \in \homotopy{f}$.  If the
map we obtain at this step is a single map then we carry on, using the
dual argument if necessary, to obtain further elements of
$\homotopy{f}$. The algorithm terminates when we reach one of the
following three cases.
\begin{itemize}
\item We reach a single map $g$ for which one of the conditions \None,
  \Ntwo\ or \Nthree\ is satisfied. In this case $f$ is null-homotopic.
\item We reach a single map $g$ which is not equal to any of the
  unused differentials with respect to the already constructed
  elements of $\homotopy{f}$. This places us in case \eqref{item:single-end} of the proposition.
\item We reach a double map; here there are insufficiently many unused
  differentials to continue to use to construct a homotopy. This
  places us in case \eqref{item:left-double-end} or
  \eqref{item:right-double-end} of the proposition.
\end{itemize}
If $f$ can also be built from the target differential, we must now
return to $f$ and carry out the dual algorithm.
\end{proof}

\begin{remark}
  Null-homotopic single and double maps correspond to `quasi-graph
  maps' which satisfy one of the graph map endpoint conditions.
\end{remark}

\begin{example}
  Recall the quasi-graph map exhibited in Example~\ref{ex: homotopy
    class}. Below is an explicit homotopy from $(c)$ to $(dc)$; the
  top line shows this at the level of complexes, the bottom line at
  the level of unfolded diagrams. In particular, the bottom line
  indicates how to use a quasi-graph map to construct families of
  homotopic maps.
\begin{align*}
& \xymatrix@R=6mm@C=7mm{     &     \smxy{P(0)} \ar[dl]_{1} \ar[d]_-{\left( \begin{smallmatrix}c & 0 \end{smallmatrix} \right)} \ar[r]^{f}      &      \smxy{P(3)} \ar[dl]_{\left( \begin{smallmatrix}0 & -1 \end{smallmatrix} \right)} \ar[d]_{0} 
\ar[r]^e     &      \smxy{P(4)} \ar[dl]_1 \ar[d]_{dc} \ar[r]^{dc}         &       \smxy{P(2)} \ar[dl]_0 \ar[d]_0 \ar[r]^-{\left( \begin{smallmatrix}b & 0 \end{smallmatrix} \right)}         &         \smxy{P(1) \oplus P(4)} \ar[dl]_0 \ar[d]_0 \ar[r]^{\left( \begin{smallmatrix}a \\ d \end{smallmatrix} \right)}          &          \smxy{P(0)} \ar[dl]_0 \\
\smxy{P(0)} \ar[r]_-{\left(\begin{smallmatrix} c & f \end{smallmatrix}\right)}          &         \smxy{P(2)\oplus P(3)} \ar[r]_-{\left(\begin{smallmatrix} 0 \\ e \end{smallmatrix}\right)}               &          \smxy{P(4)} \ar[r]_-{dc} & \smxy{P(2)}\ar[r]_b & \smxy{P(1)} \ar[r]_a & \smxy{P(0)}      &
} \\
& \xymatrix@R=2mm{ 
Q_v  \colon     &    & \xydot \ar@{=}[d]_1 \ar[dl]_c \ar[r]^-{f} & \xydot \ar@{=}[d]_{-1} \ar[r]^{e} & \xydot \ar@{=}[d]_1 \ar[r]^-{dc} \ar[dr]^>>>>{dc} & \xydot \ar[r]^{b} \ar@{==}[d]^0 & \xydot \ar@{==}[d]^0 \ar[r]^{a} & \xydot \ar@{==}[d]^0 \ar@{<-}[r]^{d} & \xydot \,  \\
\Sigma^{-1}Q_w \colon      &  \xydot \ar@{<-}[r]_-{c} & \xydot \ar[r]_-{f}  & \xydot \ar[r]_{e}                             & \xydot \ar[r]_{dc}            & \xydot \ar[r]_{b}  & \xydot \ar[r]_{a}            & \xydot         &  
}
\end{align*}
 
\end{example}

\subsection{Basic maps $f$ such that $\homotopy{f} =
  \{f\}$} \label{sec:singletons}

We start by observing that graph maps belong to singleton homotopy
classes, and thus are never null-homotopic.

\begin{lemma} \label{lem:homotopy-graph} Suppose $f \in
  \graph{v}{w}$. Then $\homotopy{f} = \{f\}$.
\end{lemma}

\begin{proof}
  Suppose $f \simeq g \colon Q_v \to Q_w$. Since $f$ is a graph map,
  there is an unfolded diagram:
\[ 
\xymatrix@R=4mm{ 
Q_v  \colon     &  \ar@{.}[r] & \xydot \ar[d]_{f_L} \ar@{-}[r]^{v_L} & \xydot \ar@{}[dl] \ar@{=}[d]_-{f_p} \ar@{-}[r]^{u_p} & \xydot \ar@{=}[d]_-{f_{p-1}} \ar@{-}[r]^{u_{p-1}} & \cdots \ar@{-}[r]^{u_2}  & \xydot \ar@{=}[d]^-{f_1} \ar@{-}[r]^{u_1} & \xydot \ar@{=}[d]^-{f_0} \ar@{-}[r]^{v_R} & \xydot \ar[d]^{f_R} \ar@{.}[r] &  \\
Q_w \colon      &  \ar@{.}[r] & \xydot \ar@{-}[r]_{w_L}             & \xydot \ar@{-}[r]_{u_p}                        & \xydot \ar@{-}[r]_{u_{p-1}}            & \cdots \ar@{-}[r]_{u_2}  & \xydot \ar@{-}[r]_{u_1}            & \xydot \ar@{-}[r]_{w_R}             & \xydot \ar@{.}[r]            &  \\
}
\]
Without loss of generality, assume that $f_i = 1$ for $0 \leq i \leq
p$. Consider the component $f_i$. Denote the corresponding component
of the map $g$ by $g_i$, which may be zero. Existence of a homotopy
between $f$ and $g$ means that the difference between $f_i$ and $g_i$
is a linear combination,
\[
f_i - g_i = 1 - g_i = \alpha u_i a + \beta u_{i+1} b + \gamma c u_i + \delta d u_{i+1}, 
\]
where the $\alpha,\beta,\gamma, \delta$ are scalars and the $a,b,c,d$
are paths in the quiver corresponding to the homotopy maps. If the
composition does not make sense, we take the corresponding scalar to
be zero. Since components of the differential are never zero, the
compositions $u_i a$, $u_{i+1} b$, $c u_i$ and $ d u_{i+1}$ are either
zero or non-stationary paths in the quiver.  It follows that $\alpha =
\beta = \gamma = \delta = 0$ and $g_i = 1$.  Lemma~\ref{lemma: basis
  maps determined} gives $f=g$, whence $\homotopy{f} = \{f\}$.  If
$i=p$ then this argument should be adjusted in the obvious way.
\end{proof}

Next we consider singleton homotopy classes of single and double maps.
It is clear that a single map $f \in \single{v}{w}$ is in a singleton
homotopy class exactly when its unfolded diagram corresponds to one of
$(i) - (iv)$ in Definition~\ref{def:singleton}. Therefore, we need
only examine when a double map occurs in a singleton homotopy class.

Recall the setup from Section~\ref{sec:single}\eqref{eq:double} on
page~\pageref{eq:double}. We say that $f = (f_L, f_R)$ has \emph{no
  common substring} if $v_0 = f_L f'$ and $w_0 = f' f_R$, and has
\emph{common substring} $s$ if $f_L = v_0 s$ and $f_R = s w_0$.  Note
that these are the only ways in which the commutative square $(*)$ in
diagram \eqref{eq:double} can decompose. Furthermore, $s$ may be a
stationary path, in which case $f$ has \emph{trivial common
  substring}.

The following lemma shows that any double map in a singleton homotopy
class satisfies condition \D\ of
Definition~\ref{def:singleton-double}. This then completes the proof
of Theorem~\ref{thm:main}.

\begin{lemma} \label{lem:doubles} Let $f=(f_L,f_R)$ be a double map.
  \begin{enumerate}[label=(\arabic*)]
  \item If $\homotopy{f} \neq \{f\}$, then $f$ has a common substring.
  \item If $f$ has a non-trivial common substring, then $f$ is
    null-homotopic.
  \end{enumerate}
\end{lemma}

\begin{proof}
  Suppose $\homotopy{f}$ is not a singleton set.  Then one of the
  following must hold: $v_Lv_L' = f_L$, $v_0v_0' = f_L$ or $w_L'w_L =
  f_L$ for some $v_L', v_0', w_L'$ paths in the quiver.  If $v_0v_0' =
  f_L$, then since $f_Lw_0 = v_0f_R$ it follows that $v_0$ and $v_L$
  start with the same arrow, a contradiction.  If $w_L'w_L = f_L$,
  then since $f_Lw_0 \neq 0$, we must have that $w_Lw_0 \neq 0$, a
  contradiction.  Thus $v_0v_0' = f_L$ and $v_0'$ is a non-zero
  component in the homotopy.  But then $v_0'w_0 = f_R$ and so $f$ has
  a common substring $v_0$.

  Let $s$ be a non-trivial common substring for $f$, then we can take
  the required family of maps to be zero everywhere except for the
  component $P(s(v_0)) \to P(t(w_0))$ which is taken to be $s$.
\end{proof}

\section{Higher-dimensional band complexes}\label{sec:bands}

Each pair $(w,\lambda) \in \bands \times \kk^*$ determines a
homogeneous tube in $\sK \subset \sD$:
\[
\xymatrix{
B_{w,\lambda,1} \ar@/^/[r] & B_{w,\lambda,2} \ar@/^/[r] \ar@/^/[l] & B_{w,\lambda,3} \ar@/^/[r] \ar@/^/[l] & B_{w,\lambda,4} \ar@/^/[r] \ar@/^/[l] & \cdots , \ar@/^/[l]
}
\]
where we refer the reader to Section~\ref{sec:higher-defn} for a
precise definition of the higher dimensional band $B_{w,\lambda,r}$
for $r >1$. We shall show that the dimensions of the Hom spaces
involving a higher dimensional band complex can be determined using
the dimension of the Hom space of the corresponding one-dimensional
band occurring at the mouth of the tube.

We start by making these definitions precise and describing unfolded
diagrams for higher dimensional tubes. For simplicity, in this section
we shall assume that $\kk$ is an algebraically closed field.

\subsection{Definition, example and unfolded
  diagrams} \label{sec:higher-defn} Let $(w,\lambda,r) \in \bands
\times \kk^* \times \bN$ and recall that $B_{w,\lambda,1}\coloneqq
(B^i_{w,\lambda,1}, D^i)$, with $D^i = (d^i_{jk})$, denotes the
one-dimensional band complex. The \emph{$r$-dimensional band complex}
is defined as follows:
\[
B_{w,\lambda,r} \coloneqq \big( (B_{w,\lambda,1}^i)^r, D^i_{w,\lambda,r} \big),
\]
where
\[
D^i_{w,\lambda,r} =\left(
\begin{array}{cccc}
D^i & A^i & & 0 \\ 
0 & D^i & & 0 \\
0 & 0 & \ddots & A^i \\
0 & 0 & & D^i \\
\end{array}
\right),
\text{ and }
A^i = (a^i_{kl})
\text{ is given by }
a^i_{kl} =
\left\{
\begin{array}{ll}
w_1 & \text{ if } d^i_{kl} = \lambda w_1; \\
0   & \text{otherwise.}
\end{array}
\right.
\]

The \emph{unfolded diagram of $B_{w,\lambda,r}$} consists of $r$
aligned copies of the unfolded diagram of $B_{w,\lambda,1}$ arranged
from top to bottom of the page called \emph{layers}, which are
connected by \emph{downwards} arrows corresponding to the non-zero
entries of $A^i$, called \emph{links}.

\begin{example}\label{ex:higherband}
  Let $z$ be the homotopy band from the Running Example.  The
  two-dimensional band complex $B_{z,\lambda,2}$ is
  \[
  \xymatrix@C=5pc{
    (P(0))^2 \ar[r]_-{\begin{pmat}c&f&0&0\\0&0&c&f\end{pmat}}
    & (P(2) \oplus P(3))^2
    \ar[r]_-{\begin{pmat}b&0&0&0\\0&e&0&0\\0&0&b&0\\
        0&0&0&e\end{pmat}}
    & (P(1) \oplus P(4))^2 \ar[r]_-{\begin{pmat}
        \lambda a&a\\d&0\\0& \lambda a\\0&d\end{pmat}}
    & (P(0))^2  \\
  } \,.
  \]
The corresponding unfolded diagram is:
\[
\xymatrix@R=1.4pc{
  \cdots \ar[dr]^a \ar[r]^{\lambda a}& \xydot & \xydot \ar[l]_{d} & \xydot \ar[l]_{e}
  & \xydot \ar[l]_{f} \ar[r]^{c} & \xydot
  \ar[r]^{b} & \xydot \ar[r]^{\lambda a} \ar[dr]^a& \xydot & \cdots
  \ar[l]_{d} 
  & \text{\small layer $2$} \\
  \cdots \ar[r]^{\lambda a}& \xydot & \xydot \ar[l]_{d} & \xydot \ar[l]_{e}
  & \xydot \ar[l]_{f} \ar[r]^{c} & \xydot
  \ar[r]^{b} & \xydot \ar[r]^{\lambda a}& \xydot & \cdots \ar[l]_{d}
  & \text{\small layer $1$}\\
}
\]

\end{example}

\subsection{Passing through the link}

\begin{definition}\label{def:lifted}
  Let $1 \leq m \leq r$ and $1 \leq n \leq s$.  We say a map $f \in
  \Hom_{\sD}(B_{v,\lambda,r},B_{w,\mu,s})$ is \emph{lifted from a map
    $f' \in \Hom_{\sD}(B_{v,\lambda,1},B_{w,\mu,1})$ to the pair
    $(m,n)$} if the components of $f'$ from layer $m$ of
  $B_{v,\lambda,r}$ to layer $n$ of $B_{w,\mu,s}$ are exactly the same
  as the components of $f'$ and $f$ is the minimal such map in terms
  of number of non-zero components.
\end{definition}

For $f' \in \Hom_{\sD}(B_{v,\lambda,1},B_{w,\mu,1})$, we shall count
the number of (homotopy classes of) maps in
$\Hom_{\sD}(B_{v,\lambda,r},B_{w,\mu,s})$ which are lifted from $f'$.
The idea is to put a copy of $f'$ between the pair $(m,n)$ of layers
and see if a map arises; such a map will be called a \emph{candidate
  map}.  If a component of $f'$ composes non-trivially with a link
arrow, we say that the (candidate) map \emph{passes through the link}.
This will cause lifted maps to have non-zero components between more
layers than just the pair $(m,n)$.

\begin{lemma}\label{lem:differentialhigher}
  If $f\in \Hom_{\sD}(B_{v,\lambda,r}, B_{w,\mu,s})$ then there is a
  map $f' \in \Hom_{\sD}(B_{v,\lambda,1}, B_{w,\mu,1})$ such that $f$
  is lifted from $f'$ to a pair $(m,n)$.
\end{lemma}

\begin{proof}
  As with one-dimensional maps, all maps in
  $\Hom_{\sD}(B_{v,\lambda,r}, B_{w,\mu,s})$ are completely determined
  by any of their non-zero components.  Ignoring the link arrows
  between layers, we simply have $r$ copies of the band in
  $B_{v,\lambda,r}$ and $s$ copies in $B_{w,\mu,s}$. It follows that
  if there is a map $f \colon B_{v,\lambda,r} \to B_{w,\mu,s}$ with a
  non-zero component from layer $m'$ of $B_{v,\lambda, r}$ to layer
  $n'$ of $B_{w,\mu, s}$, then there is a map $f' \in
  B_{v,\lambda,1}\to B_{w,\mu, 1}$ with the same non-zero component
  and $f$ is lifted from $f'$.
\end{proof}

The following lemma is straightforward.

\begin{lemma}\label{lem:nolink}
  Suppose $f \in \Hom_{\sD}(B_{v,\lambda,r}, B_{w,\mu,s})$ is lifted
  from $f'\in \Hom_{\sD}(B_{v,\lambda,1}, B_{w,\mu,1})$ to the pair of
  layers $(m,n)$.  Then:
\begin{enumerate}
\item \label{item:single} if $f'$ is a single map, then $f$ does not
  pass through any link;
\item \label{item:m=r} if $m=r$, then $f$ does not pass through a link
  in $B_{v,\lambda,r}$;
\item \label{item:n=1} if $n=1$, then $f$ does not pass through a link in $B_{w,\mu,s}$.
\end{enumerate}
\end{lemma}

Recall the notation in Notation~\ref{not:maps}. To take care of
homotopies for higher-dimensional homotopy bands, we need to modify
Definition~\ref{def:homotopy-class} slightly. Recall that the link in
$B_{v,\lambda,r}$ is given by the homotopy letter $v_1$.

\begin{definition} \label{def:homotopy-class-bands} Let $f \colon
  B_{v,\lambda,r} \rightarrow B_{w,\mu,s}$ be a map lifted from a map
  $f'$ to the pair $(m,n)$ as in Definition \ref{def:lifted}.  We
  denote by $\homotopylayers{m}{n}{f}$ the set of $\kk$-linear
  combinations $\rho_1 g_1 + \rho_2 (\widetilde{v_1}) + \rho_3
  (\widetilde{w_1})$ of maps, such that $\rho_1 g_1 + \rho_2
  (\widetilde{v_1}) + \rho_3 (\widetilde{w_1}) \simeq f$ with $\rho_1
  g_1 \in \homotopy{f'}$ and $(\widetilde{v_1})$, $(\widetilde{w_1})$
  are the single maps $(v_1)$, $(w_1)$ lifted to pairs $(m+1,n)$ and
  $(m,n-1)$, respectively. When $m$ and $n$ are understood, we simply
  write $\homotopyband{f}$ for $\homotopylayers{m}{n}{f}$.
\end{definition}

  Note that $\rho_2 \neq 0$ if and only if the homotopy map
  passes through the link in $B_{v,\lambda,r}$; similarly for
  $\rho_3$.  Thus, the homotopy class of $f$ can be determined
  by $\homotopyband{f}$ and $\homotopy{f'}$.

\subsection{A worked example} \label{sec:example}

Before discussing the general behaviour of maps involving
higher-dimensional homotopy bands, it is useful to examine an example
in detail. This example will exhibit all possibilities regarding
lifting of maps and homotopy classes and clarify the strategy in the
proofs of the general results.

Throughout this worked example, $\Lambda$ will be given by the
following bound quiver.

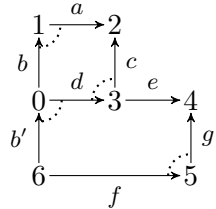
\[
\begin{tikzpicture}
  \node (0) at (0,0) [smallvertex] {$0$};
  \node (1) at (0,1) [smallvertex] {$1$};
  \node (2) at (1,1) [smallvertex] {$2$};
  \node (3) at (1,0) [smallvertex] {$3$};
  \node (4) at (2,0) [smallvertex] {$4$};
  \node (5) at (2,-1) [smallvertex] {$5$};
  \node (6) at (0,-1) [smallvertex] {$6$};
  \draw [->] (0) -- node [left, tinyvertex] {$b$} (1);
  \draw [->] (1) -- node [above, tinyvertex] {$a$} (2);
  \draw [->] (3) -- node [right, tinyvertex] {$c$} (2);
  \draw [->] (0) -- node [above, tinyvertex] {$d$} (3);
  \draw [->] (3) -- node [above, tinyvertex] {$e$} (4);
  \draw [->] (6) -- node [below, tinyvertex] {$f$} (5);
  \draw [->] (5) -- node [right, tinyvertex] {$g$} (4);
  \draw [->] (6) -- node [left, tinyvertex] {$b'$} (0);
  \draw[dotted,thick] (0,0.7) arc (-90:0:8pt);
  \draw[dotted,thick] (0.7,0) arc (180:90:8pt);
  \draw[dotted,thick] (1.7,-1) arc (180:90:8pt);
  \draw[dotted,thick] (0.3,0) arc (0:-90:8pt);
\end{tikzpicture}
\]

\noindent We consider the homotopy bands $v =
(e,0,1)(c,1,0)(a,0,1)(bb',1,2)(f,2,1)(g,1,0)$ and $w = (d, i, i-1)(c,
i-1,i-2)(a,i-2,i-1)(b,i-1,i)$, where the degree $i$ will be specified
by the diagrams occurring in each example in the context of the
particular map or homotopy class we are interested in.

Our first example indicates the typical situation of lifting a singleton homotopy class.

\begin{example}\label{ex:bands-example1}
  The candidate map is a graph map $h' \in
  \Hom(B_{w,\lambda,1},B_{v,\mu,1})$ lifted to the pair $(1,1)$ of
  layers in $\Hom(B_{w,\lambda,2},B_{v,\mu,1})$.  The components of
  the lifted map $h$ which are forced by passing through the link are
  drawn as broken lines (they are all identities).
  
  
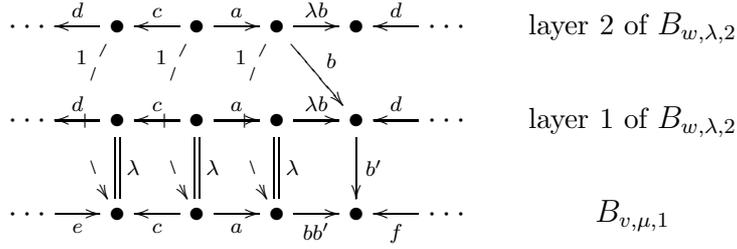
\[
  \xymatrix@=1.4pc{ \cdots & \xydot \ar[l]_-{d}
    \ar@{-->}@/_1pc/[dd]_(0.2){1} & \xydot \ar[l]_-{c} \ar[r]^-{a}
    \ar@{-->}@/_1pc/[dd]_(0.2){1} & \xydot \ar[r]^-{\lambda b}
    \ar[dr]^-{b} \ar@{-->}@/_1pc/[dd]_(0.2){1} &\xydot & \cdots
    \ar[l]_-{d}
    & {\small\text{layer 2 of }B_{w,\lambda,2}}\\
    \cdots & \xydot \ar[l]_-{d} \ar@{=}[d]^{\lambda} & \xydot
    \ar[l]_-{c} \ar[r]^-{a} \ar@{=}[d]^{\lambda} & \xydot
    \ar[r]^-{\lambda b} \ar@{=}[d]^{\lambda} & \xydot \ar[d]^-{b'}
    & \cdots \ar[l]_-{d} & {\small\text{layer 1 of }B_{w,\lambda,2}}\\
    \cdots \ar[r]_-{e} & \xydot & \xydot \ar[l]^-{c} \ar[r]_-{a} &
    \xydot \ar[r]_-{bb'} & \xydot & \cdots \ar[l]^-{f} & {\small B_{v,\mu,1}}\\
  }
  \]
  Note that $h'$ also lifts to a map which includes a copy of $h'$
  from layer 2 of $B_{w,\lambda,2}$ to the unique layer of
  $B_{v,\mu,1}$. Thus, the graph map $h'$ lifts to two maps in
  $\Hom_{\sD}(B_{w,\lambda,2},B_{v,\mu,1})$.
\end{example}

Our next example examines the case of lifting a non-singleton homotopy
class. Recall the notation for homotopy equivalent basis maps in
Definition~\ref{def:homotopy-class-bands}.

\begin{example}\label{ex:bands-example2}
  We consider the homotopy set $\homotopy{(b)} =
  \{(b),-(\frac{1}{\mu}a),(\frac{1}{\mu}c),-(\frac{1}{\mu}e,\frac{1}{\mu}d)
  \}$ in $\Hom_{\sC}(B_{v,\lambda,1},B_{w,\mu,1})$, contributing one
  map in $\Hom_{\sD}(B_{v,\lambda,1},B_{w,\mu,1})$, which is indicated
  in the following diagram (drawn without signs or scalars).
 
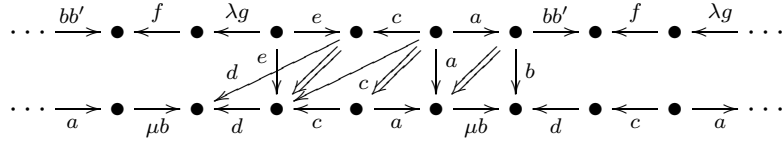
\[
  \xymatrix@=1.4pc{ \cdots \ar[r]^{bb'} & \xydot & \xydot \ar[l]_f &
    \xydot \ar[l]_{\lambda g} \ar[r]^e \ar[d]_<<e & \xydot \ar@{=>}[dl]
    \ar[dll]_>>>>d & \xydot \ar[l]_c \ar[d]^(0.4)a \ar@{=>}[dl]
    \ar[r]^a \ar[dll]^c & \xydot \ar[r]^{bb'} \ar[d]^b \ar@{=>}[dl] &
    \xydot & \xydot \ar[l]_f & \cdots \ar[l]_{\lambda g}
    \\
    \cdots \ar[r]_a & \xydot \ar[r]_{\mu b} & \xydot & \xydot \ar[l]^d
    & \xydot \ar[l]^c \ar[r]_a &\xydot \ar[r]_{\mu b} &\xydot & \xydot
    \ar[l]^d & \xydot \ar[l]^c \ar[r]_a & \cdots 
    \\
  }
  \]

  We shall now lift $\homotopy{(b)}$ to
  $\Hom_{\sD}(B_{v,\lambda,1},B_{w,\mu,2})$.  From
  Lemma~\ref{lem:nolink}\eqref{item:single} it is clear that each map
  in $\homotopy{(b)}$ gives rise to two maps in
  $\Hom_{\sC}(B_{v,\lambda,1},B_{w,\mu,2})$ without adding any extra
  components.  For convenience, decorate maps to the first layer of
  $B_{w,\mu,2}$ with a tilde, i.e. $\tilde{b}$ and so on, and maps to
  the second layer with a hat, i.e. $\hat{b}$ and so on.

  By Lemma~\ref{lem:nolink}, we get $\homotopyband{(\tilde{b})}=
  \{(\tilde{b}),(-\frac{1}{\mu}\tilde{a}),(\frac{1}{\mu}\tilde{c}),
  -\frac{1}{\mu}(\tilde{e},\tilde{d})\}$.  We use the following
  diagram to determine $\homotopyband{(\hat{b})}$.
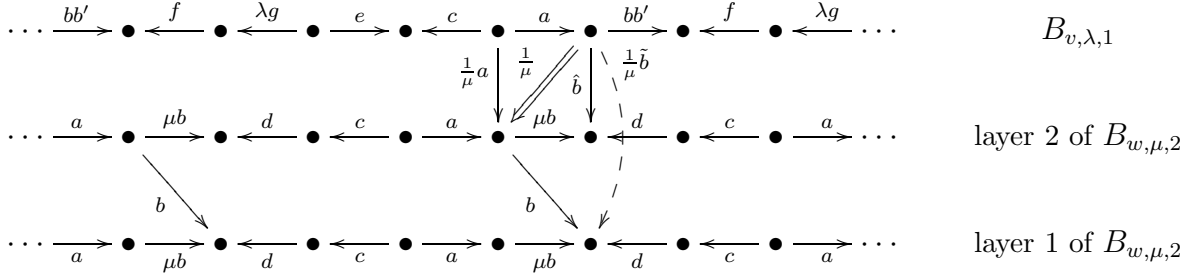
  \[
  \xymatrix@=1.8pc{\cdots \ar[r]^{bb'} & \xydot & \xydot \ar[l]_f &
    \xydot \ar[l]_{\lambda g} \ar[r]^e & \xydot & \xydot \ar[l]_c
    \ar[d]_(0.4){\frac{1}{\mu}a} \ar[r]^a & \xydot \ar[r]^{bb'}
    \ar[d]_{\hat b} \ar@{=>}[dl]_{\frac{1}{\mu}}
    \ar@/^1pc/@{-->}[dd]^(0.2){\frac{1}{\mu} \tilde b} & \xydot &
    \xydot \ar[l]_f & \cdots
    \ar[l]_{\lambda g} & {\small B_{v,\lambda,1}}\\
    \cdots \ar[r]^a & \xydot \ar[dr]_{b} \ar[r]^{\mu b} & \xydot &
    \xydot \ar[l]_-d & \xydot \ar[l]_c \ar[r]^a &\xydot \ar[r]^{\mu b}
    \ar[dr]_{b} &\xydot & \xydot \ar[l]_d & \xydot
    \ar[l]_c \ar[r]^a & \cdots & {\small\text{layer 2 of }B_{w,\mu,2}}\\
    \cdots \ar[r]_a & \xydot \ar[r]_{\mu b} & \xydot & \xydot \ar[l]^d
    & \xydot \ar[l]^c \ar[r]_a &\xydot \ar[r]_{\mu b} &\xydot & \xydot
    \ar[l]^d & \xydot \ar[l]^c \ar[r]_a & \cdots
    & {\small\text{layer 1 of }B_{w,\mu,2}}\\
  }
  \]
  The homotopy passes through the link, so we have $(\hat b) \simeq
  (-\frac{1}{\mu} \hat a) - (\frac{1}{\mu}\tilde{b}$).  It is easy to
  see that $(\hat a) \simeq (-\hat c) \simeq (\hat e, \hat d) \simeq
  (-\mu b)$.  Therefore, we determine $\homotopyband{(\hat b)}$ and
  $\homotopyband{(\hat a)}$ as follows:
\begin{align*}
\homotopyband{(\hat b)} & =
  \{ (\hat b), -(\frac{1}{\mu} \hat a) - (\frac{1}{\mu}\tilde{b}),
  (\frac{1}{\mu} \hat c) - (\frac{1}{\mu}\tilde{b}), -\frac{1}{\mu}
  (\hat e, \hat d) - (\frac{1}{\mu}\tilde{b}) \} \\
  \homotopyband{(\hat a)} & = 
  \{ (\hat a),  -(\hat c),  (\hat{e},\hat{d}), 
  - (\mu \hat{b}) - (\tilde{b})
  \} \, .
\end{align*}
  
Now, we have three different homotopy classes, but each of them is a
$\kk$-linear combination of the other two.  This shows that such a
homotopy class lifts to two homotopy classes of maps in
$\Hom_{\sC}(B_{v,\lambda,1},B_{w,\mu,2})$, thus giving rise to two
maps in $\Hom_{\sD}(B_{v,\lambda,1},B_{w,\mu,2})$.
\end{example}

In the next example we look at what happens when one tries to lift an
isomorphism. The following example shows that one cannot lift an
identity morphism on a one-dimensional homotopy band to every pair of
layers $(m,n)$.

\begin{example}\label{ex: bands-example3}
  Consider the homotopy band $w$ above and let $\lambda \in \kk^*$.
  Taking a copy of the identity from layer 1 of $B_{w,\lambda, 2}$ to
  $B_{w,\lambda, 1}$, we are forced to take the dashed components as
  before. However, once we reach the end of the band we must add
  dashed arrows to the right-hand side of the link as well:
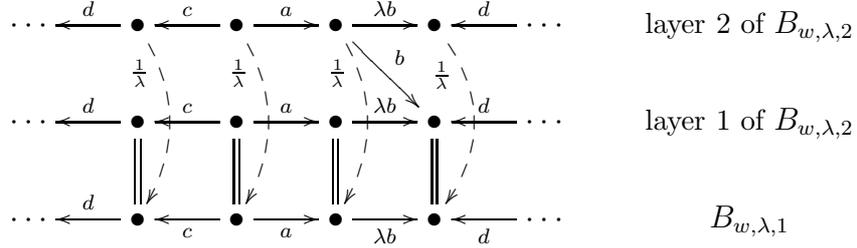
\[
\xymatrix@R=1.5pc{
  \cdots & \xydot \ar[l]_-{d} \ar@{-->}@/^1pc/[dd]_(0.2){\frac{1}{\lambda}}
  & \xydot \ar[l]_-{c} \ar[r]^-{a} \ar@{-->}@/^1pc/[dd]_(0.2){\frac{1}{\lambda}} &
  \xydot \ar[r]^-{\lambda b} \ar[dr]^-{b}
  \ar@{-->}@/^1pc/[dd]_(0.2){\frac{1}{\lambda}}
  &\xydot \ar@{-->}@/^1.25pc/[dd]_(0.2){\frac{1}{\lambda}}  & \cdots \ar[l]_-{d}
  & {\small\text{layer 2 of }B_{w,\lambda,2}}\\
  \cdots & \xydot \ar[l]_-{d} \ar@{=}[d] & \xydot \ar[l]_-{c}
  \ar[r]^-{a} \ar@{=}[d] &
  \xydot \ar[r]^-{\lambda b} \ar@{=}[d] & \xydot \ar@{=}[d]
  & \cdots \ar[l]_-{d} & {\small\text{layer 1 of }B_{w,\lambda,2}}\\
  \cdots  & \xydot \ar[l]_-{d}& \xydot \ar[l]^-{c} \ar[r]_-{a} &
  \xydot \ar[r]_-{\lambda b} & \xydot & \cdots \ar[l]^-{d} & {\small B_{w,\lambda,1}}\\
}
\] 
but then the square involving the link does not commute: we have $2b +
\lambda b = b +\lambda b$. Therefore, we cannot lift a copy of the
identity to layer 1 in $B_{w,\lambda,1}$ and get a well-defined map of
homotopy band complexes.
\end{example}

In the final example of this section, we consider the homotopy class
arising from the identity map on a one-dimensional band complex.

\begin{example}\label{ex:shift}
  Consider the complex $B_{w,\lambda,1}$.  The identity map on
  $B_{w,\lambda,1}$ gives rise to a homotopy class in
  $\Hom_{\sD}(B_{w,\lambda,1},\Sigma B_{w,\lambda,1})$ which is
  non-zero:
  
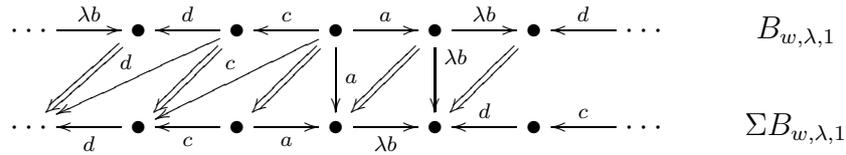
\[
  \xymatrix@R=1.5pc{
    \cdots \ar[r]^-{\lambda b} & \xydot \ar@{=>}[dl] 
    & \xydot \ar@{=>}[dl] \ar[l]_-{d} \ar[dll]_{d}
    & \xydot \ar@{=>}[dl] \ar[l]_-{c} \ar[r]^-{a} \ar[dll]_{c} \ar[d]^a
    & \xydot \ar@{=>}[dl] \ar[r]^-{\lambda b} \ar[d]^<<{\lambda b}
    &\xydot \ar@{=>}[dl]  & \cdots \ar[l]_-{d}
    & {\small B_{w,\lambda,1}}\\
    \cdots & \xydot \ar[l]^-{d} & \xydot \ar[l]^-{c}
    \ar[r]_-{a} &
    \xydot \ar[r]_-{\lambda b} & \xydot & \xydot \ar[l]_-{d} &
    \ar[l]_-{c} \cdots
    & {\small \Sigma B_{w,\lambda,1}}\\
  }
  \]
  We have $\homotopy{(b)} =
  \{(b),-\lambda^{-1}(a),\lambda^{-1}(c),-\lambda^{-1}(d),(b)\}$,
  where we have written the $(b)$ twice to emphasise the following key
  point: this gives a non-trivial way in which to obtain the
  tautologous homotopy equivalence $b - b \simeq 0$. Note that there
  can be other (homotopy classes of) maps in
  $\Hom(B_{w,\lambda,1},\Sigma B_{w,\lambda,1})$; these behave as in
  Examples \ref{ex:bands-example1} and \ref{ex:bands-example2}.
\end{example}

\subsection{Maps which are not self-extensions}

In this section we shall make a number of statements regarding
dimensions of Hom spaces. It is useful to first set up some
notation. Let $\Rad_{\sD}(\cxP,\cxQ)$ denote the space of
non-isomorphisms $\cxP \to \cxQ$. Following standard notation in
algebraic geometry, we write
\[
\hom_{\sD}(\cxP,\cxQ) = \dim \Hom_{\sD}(\cxP,\cxQ)
\quad \text{and} \quad
\rad_{\sD}(\cxP,\cxQ) = \dim \Rad_{\sD}(\cxP,\cxQ).
\]

We now state the main result of this section.

\begin{theorem}\label{thm:homotopy-bands}  
Let $r, s \in \bN$, $w, v \in \bands $, $u \in \strings$ and $\lambda, \mu \in \kk^*$ be such that
$B_{v,\lambda,1} \ncong \Sigma  B_{w,\mu,1}$.  Let $\delta_{B_{v,\lambda,1},B_{w,\mu,1}}$ be the Kronecker delta. Then
\begin{enumerate}[label=(\arabic*)]
\item $\hom_{\sD}(B_{v,\lambda,r},B_{w,\mu,s}) = \min\{r,s\} \cdot \delta_{B_{v,\lambda,1},B_{w,\mu,1}} + rs \cdot \rad(B_{v,\lambda,1},B_{w,\mu,1})$;
\item $\hom_{\sD}(B_{v,\lambda,r},P_u) = r \cdot \hom_{\sD}(B_{v,\lambda,1},P_u)
\ \text{and} \
\hom_{\sD}(P_u,B_{v,\lambda,r}) = r \cdot \hom_{\sD}(P_u,B_{v,\lambda,1})$.
\end{enumerate}
\end{theorem}

We now prove the first assertion of Theorem~\ref{thm:homotopy-bands}
in a sequence of lemmas; the second assertion is proved
similarly. From now on assume that $v, w \in \bands$, $\lambda, \mu
\in \kk^*$ and $B_{v,\lambda,1} \ncong \Sigma B_{w,\mu,1}$.

\begin{lemma}\label{lem:singleton-lift}
  Let $0\neq f' \in \Hom_{\sD}(B_{v,\lambda,1},B_{w,\mu,1})$. If
  $\homotopy{f'} = \{f'\}$ and $f'$ is not an isomorphism, then $f'$
  can be lifted to any pair of layers in $B_{v,\lambda,r}$ and
  $B_{w,\mu,s}$.
\end{lemma}

\begin{proof}
  This follows directly from Lemma~\ref{lem:nolink}, noting that the
  resulting maps may acquire extra components if they pass through a
  link.
\end{proof}

The following lemmas deal with the generic case of
Theorem~\ref{thm:homotopy-bands}; the example to bear in mind is
Example~\ref{ex:bands-example2}.

\begin{lemma} \label{lem:lift-lin-comb} Let $v$ and $w$ be homotopy
  bands, and $\lambda, \mu \in \kk^*$ such that $B_{v,\lambda,1} \neq
  \Sigma B_{w,\mu,1}$ and $B_{v,\lambda,1} \neq B_{w,\mu,1}$.  Let
  $f^{(m,n)} \colon B_{v,\lambda,1} \to B_{w,\mu,1}$ be a lift of $f
  \in \single{v}{w} \cup \double{v}{w}$, which is not null-homotopic,
  to the layers $(m,n)$. Consider the lift $g^{(m,n)}$ of any map $g
  \in \homotopy{f}$. Then $g^{(m,n)}$ is homotopic to linear
  combination of representatives from only $\homotopylayers{i}{j}{f}$
  for $i \geq m$ and $j \leq n$.
\end{lemma}

\begin{proof}
  Choose the quasi-graph map in $\quasi{v}{w}$ of which $f$ is a
  representative. Starting from the left of the quasi-graph map, we
  write down the homotopy set
\[
\homotopy{f_1} = \{f_1,\ldots,\lambda f_p, \ldots,\lambda \mu f_q,
\ldots, \lambda \mu f_k\}.
\]
Without loss of generality, we may assume $f=f_1$; taking a different
choice simply requires multiplication by the appropriate scalar. Note
that our definition above places the link in the source band at $f_p$
and in the target band at $f_q$ with $p<q$. Other choices of $p$ and
$q$ can be considered analogously. We consider the following lifts of
$\homotopy{f_1}$.
\[
\homotopylayers{i}{j}{f_1}  =  \{\tilde{f}_1^{(i,j)},\ldots,\tilde{f}_p^{(i,j)},\ldots,\tilde{f}_q^{(i,j)}, \ldots, \tilde{f}_k^{(i,j)}\}
\]
where
\[
\tilde{f}_t^{(i,j)} \coloneqq
\left\{
\begin{array}{ll}
f_t^{(i,j)} & \text{ for } 1 \leq t < p; \\
\lambda f_t^{(i,j)} + f_p^{(i+1,j)} & \text{ for } p \leq t < q; \\
\lambda \mu f_t^{(i,j)} + f_p^{(i+1,j)} + f_q^{(i,j-1)} & \text{ for } q \leq t \leq k,
\end{array}
\right.
\]
where we interpret $f_p^{(i+1,j)} = 0$ for $i =r$ and $f_q^{(i,j-1)} =
0$ for $j = 1$.

The problem is to write the maps $f_t^{(i,j)}$ as linear combinations
of maps from these sets. For $f_t^{(i,j)}$ with $1 \leq t < p$, there
is nothing to show. For $f_t^{(i,j)}$ with $p \leq t < q$, we can
write this map as a linear combination of maps from the sets
$\homotopylayers{i}{j}{f_1}, \ldots, \homotopylayers{r}{j}{f_1}$. For
$q \leq t \leq k$, we get $f_t^{(i,j)}$ as a linear combination of
representatives from the sets as above and
$\homotopylayers{i}{j-1}{f_1},\ldots,\homotopylayers{i}{1}{f_1}$, as
required.
\end{proof}

\begin{lemma}\label{lemma: non-shifted, non-iso bands}
  Let $v$ and $w$ be homotopy bands, and $\lambda, \mu \in \kk^*$ such
  that $B_{v,\lambda,1} \neq \Sigma B_{w,\mu,1}$ and $B_{v,\lambda,1}
  \neq B_{w,\mu,1}$.  Then $ \hom_{\sD}(B_{v,\lambda,r}, B_{w,\mu,s})
  = rs \cdot \hom_{\sD}(B_{v,\lambda,1}, B_{w,\mu,1}) $.
\end{lemma}

\begin{proof}
  By Lemmas~\ref{lem:differentialhigher} and \ref{lem:singleton-lift},
  we need only check that we have no more than $rs$ linearly
  independent non-singleton homotopy classes. Suppose $f \in
  \single{v}{w} \cup \double{v}{w}$ is such that $\homotopy{f} \neq
  \{f\}$. If the quasi-graph map determining the representatives of
  $\homotopy{f}$ does not pass through the link then it is clear that
  $f$ contributes $rs$ homotopy classes of maps to
  $\Hom_{\sD}(B_{v,\lambda,r},B_{w,\mu,s})$. So we may assume that the
  quasi-graph map passes through the link in one of the two homotopy
  bands. Here we assume we are in the generic situation where the
  quasi-graph map passes through the link in both bands.

  Consider the homotopy sets and notation as in the statement and
  proof of Lemma~\ref{lem:lift-lin-comb} above.
  Lemma~\ref{lem:lift-lin-comb} shows that any lift of a single or
  double map homotopic to $f$ is a linear combination of
  representatives from $\homotopylayers{i}{j}{f_1}$ for $1\leq i \leq
  r$ and $1 \leq j \leq s$. We just need to establish linear
  independence.

  We claim that every lift of a map in $\homotopy{f}$ is homotopic to
  a linear combination of the $f_q^{(i,j)}$.  First note that any
  representative in the $\homotopylayers{r}{j}{f_1}$ is homotopic to a
  linear combination of the $f_q^{(r,j)}$ for $1 \leq j \leq s$ since
  $f_p^{(r+1,j)}=0$ for each $j$. Now consider any representative in
  the $\homotopylayers{r-1}{j}{f_1}$. Each of these is homotopic to a
  linear combination of the maps $f_q^{(r-1,j)}$, $f_q^{(r,j)}$ and
  $f_p^{(r,j)}$ for $1 \leq j \leq s$. However, we have just shown
  that $f_p^{(r,j)}$ is homotopic to a linear combination of the maps
  $f_q^{(r,j)}$. Continuing in this way gives the claim.

  Thus, to get linear independence, it is sufficient to show that if
  $\sum_{i,j} \alpha_{i,j} f_q^{(i,j)} \simeq 0$ then $\alpha_{i,j}=0$
  for $1 \leq i \leq r$ and $1 \leq j \leq s$. This can be seen by
  carrying out the algorithm given in the proof of
  Proposition~\ref{prop:homotopy-classes}: the required zeros are
  obtained by the fact that the quasi-graph maps do not satisfy the
  required null-homotopic (= graph map) endpoint conditions.
\end{proof}

Next we deal with special case $B_{v,\lambda,1} = B_{w,\mu,1}$.

\begin{lemma} \label{lem:iso-bands}
Let $r$ and $s$ be positive integers, $v$ a homotopy band and $\lambda \in \kk^*$.  Then
  \[
  \hom_{\sD}(B_{v,\lambda,r},B_{v,\lambda,s}) = \min\{r,s\} + rs \cdot \rad(B_{v,\lambda,1},B_{v,\lambda,1}).
  \]
\end{lemma}
  
\begin{proof} 
  By the same arguments as in the proof of Lemma \ref{lemma:
    non-shifted, non-iso bands}, any non-isomorphism from
  $B_{v,\lambda,1}$ to $B_{v,\lambda,1}$ can be lifted to any pair of
  layers $(m,n)$.  So we must have that
  $\hom_{\sD}(B_{v,\lambda,r},B_{v,\lambda,s})$ is at least $rs \cdot
  \rad(B_{v,\lambda,1},B_{v,\lambda,1})$.
  
  It remains to consider how many maps can be lifted from the identity
  map.  Clearly the identity will pass through the link but, in
  contrast to non-isomorphism case, the non-zero components acquired
  in the layers above and below $n$ and $m$ will occur on both sides
  of the link (as in Example \ref{ex: bands-example3}). Thus, if $f$
  is lifted from the identity to $(m,n)$ for $1 < m < r$ and $1 <n
  <s$, then $f$ is also the identity lifted to $(m+1, n+1)$ and to
  $(m-1, n-1)$.
  
  The identity map does not lift to $(m,s)$ for $m \neq r$ nor to $(1,
  n)$ for $n \neq 1$ because it is not possible to satisfy the
  required commutativity relations, cf. Example \ref{ex:
    bands-example3}.  It follows from this that there are
  $\min\{r,s\}$ maps lifted from the identity, each projecting from
  the top-most layers of $B_{v,\lambda,r}$ onto the bottom-most layers
  of $B_{v,\lambda,s}$.
\end{proof}

\subsection{Self-extensions of bands}

As in Lemma~\ref{lem:iso-bands}, we must be careful when considering
the homotopy class in $\Hom_{\sC}(B_{v,\lambda,1}, \Sigma
B_{v,\lambda,1})$ corresponding to the identity in
$\Hom_{\sC}(B_{v,\lambda,1},B_{v,\lambda,1})$. Here, the example to
keep in mind is Example~\ref{ex:shift}.

\begin{lemma} \label{lem:non-trivial-homotopy} Let $1 \leq i \leq r$
  and $1 \leq j \leq s$. Denote the lift of $(v_1)$ to layers $(i,j)$
  by $(v_1)^{(i,j)}$. Then, $ (v_1)^{(i,1)} \simeq 0 \text{ for } i >
  1 \text{ and } (v_1)^{(i,j-1)} \simeq -(v_1)^{(i+1,j)} \text{
    whenever } j > 1 $, where we take $(v_1)^{(r+1,j)}=0$.
\end{lemma}

\begin{proof}
  The single map $(v_1) \colon B_{v,\lambda,1} \to \Sigma
  B_{v,\lambda,1}$ is a representative of the quasi-graph map
  corresponding to the identity $B_{v,\lambda,1} \to
  B_{v,\lambda,1}$. Moreover, this quasi-graph map determines a
  non-trivial homotopy from $(v_1)$ to $(v_1)$. Lifting this
  non-trivial homotopy to the pair of layers $(i,j)$ gives $
  (v_1)^{(i,j)} \simeq (v_1)^{(i,j)} - \lambda^{-1}(v_1)^{(i,j-1)} -
  \lambda^{-1}(v_1)^{(i+1,j)} $, which outputs the homotopy
  equivalences as claimed.
\end{proof}

\begin{proposition} \label{prop:selfext} Let $r,s \in \bN$, $v \in
  \bands$ and $\lambda \in \kk^*$. Then
\[ 
\hom_{\sD}(B_{v,\lambda,r}, \Sigma B_{v,\lambda,s}) = \min\{r,s\} + rs
\cdot (\hom_{\sD}(B_{v,\lambda,1}, \Sigma B_{v,\lambda,1})-1)
\]
\end{proposition}

\begin{proof}
  Observe that every basis element of $\Hom_{\sD}(B_{v,\lambda,1},
  \Sigma B_{v,\lambda,1})$ passes through the link at most once except
  for the quasi-graph map corresponding to the shifted identity.  As
  before, each of these basis elements can be lifted to $(m,n)$ for
  any $m,n$ and so $\hom_{\sD}(B_{v,\lambda,r}, \Sigma
  B_{v,\lambda,s})$ is at least $rs \cdot (\hom_{\sD}(B_{v,\lambda,1},
  \Sigma B_{v,\lambda,1})-1)$. It remains to determine how many
  linearly independent homotopy classes are lifted from this remaining
  homotopy class.

  Repeated application of Lemma~\ref{lem:non-trivial-homotopy} yields
  the following homotopy equivalences:
  \begin{eqnarray*}
    (v_1)^{(1,1)} & \simeq & -(v_1)^{(2,2)} \simeq (v_1)^{(3,3)} \simeq \cdots \simeq \pm (v_1)^{(s,s)}; \\
    (v_1)^{(1,2)} & \simeq & -(v_1)^{(2,3)} \simeq (v_1)^{(3,4)} \simeq \cdots \simeq \pm (v_1)^{(s-1,s)}; \\
    \vdots \\
    (v_1)^{(1,s)}.
  \end{eqnarray*}
  Therefore, if $r \geq s$, then none of these homotopy equivalences
  includes a zero map, and therefore none of the homotopy classes
  above are null. In particular, there are $s$ homotopy classes. If
  $r<s$, the the first $s-r$ homotopy classes are actually
  null-homotopic, which leaves $s - (s-r) = r$ non-null homotopy
  classes.

  To see that all other lifts of $(v_1)$ are null-homotopic, we again
  apply Lemma~\ref{lem:non-trivial-homotopy}. In this case, it yields
  that $0 \simeq (v_1)^{(i,1)} \simeq -(v_1)^{(i+1,2)} \simeq
  (v_1)^{(i+2,3)} \simeq \cdots \simeq \pm (v_1)^{(r,s+2-i)}$ for $2
  \leq i \leq r$ when $r \geq s$, and $0 \simeq (v_1)^{(i,1)} \simeq
  -(v_1)^{(i+1,2)} \simeq (v_1)^{(i+2,3)} \simeq \cdots \simeq \pm
  (v_1)^{(r,r+1-i)}$ for $2 \leq i \leq r$ when $r < s$.

  All that remains to show is that lifts of all other single maps
  occurring as representatives of the quasi-graph map corresponding to
  $\id_{B_{v,\lambda,1}}$ are linear combinations of representatives
  of the homotopy classes described above. This follows directly from
  Lemma~\ref{lem:lift-lin-comb}.
\end{proof}

Many tubes occurring in representation theory are hereditary. It is
natural to ask whether the same is true for the homogeneous tubes that
arise from homotopy bands. This can be seen not to be the case using
the following straightforward example.

\begin{example}
  Consider the algebra and homotopy band from the Running Example:
\[
z = (d,3,2)(e,2,1)(f,1,0)(c,0,1)(b,1,2)(a,2,3). 
\]
The unfolded diagram below shows a graph map of degree 3, thus
$\Ext^{3}(B_{z,\lambda,1},B_{z,\lambda,1}) \neq 0$.
\[
 \xymatrix@R=1pc@C=1.4pc{
\text{degrees:} & 2                          & 3                              & 2                  & 1                          & 0                                        & 1                  & 2                        & 3                              & 2 \\
B_{z,\lambda,1} \colon \ar[d] & \xydot \ar[r]^-{\lambda a} & \xydot                         & \xydot \ar[l]_-{d} & \xydot \ar[l]_-{e}         & \xydot \ar[l]_-{f} \ar[r]^-{c} \ar@{=}[d] & \xydot \ar[r]^-{b} & \xydot \ar[r]^{\lambda a} & \xydot                         & \xydot \ar[l]_-{d}   \\
\Sigma^3 B_{z,\lambda,1} \colon  & \xydot                     & \xydot \ar[l]^-{f} \ar[r]_-{c} & \xydot \ar[r]_-{b} & \xydot \ar[r]_-{\lambda a} & \xydot                                    & \xydot \ar[l]^-{d} & \xydot \ar[l]^-{e}       & \xydot \ar[l]^-{f} \ar[r]_-{c} & \xydot \\
\text{degrees:} & -2                          & -3                              & -2                  & -1                          & 0                                        & -1                  & -2                        & -3                              & -2  
}
\]
\end{example}

\section{Application: Irreducible morphisms between string
  complexes} \label{sec:irreducible}

In this section, we recover Bobi\'nski's description \cite{Bo} of the
irreducible maps in $\sK$ and extend it to $\sD$. In \cite{Bo},
Bobi\'nski employed the Happel functor $F\colon \sD \into
\stmod{\hat{\Lambda}}$ to determine the AR structure of $\sK \into
\sD$. We present an algorithm intrinsic to $\KpL$, which allows one to
compute all irreducible maps in $\sD$. (Note that only the full
subcategory of perfect complexes $\sK \into \sD$ admits AR triangles.)
By \cite{BR}, it is known that for any $P \in \ind{\sK}$, there are at
most two irreducible maps starting and ending at $P$.

Recall from \cite{AG,Bo} that there exist \emph{string functions} $S,T
\colon \strings \to \{-1,1\}$ satisfying certain properties (see
\cite{AG,Bo} for precise details). We remark only on the consequences.
\begin{itemize}
\item We must consider formal inverses of trivial homotopy strings,
  say $(1_x,i,i)^{-1} = (\overline{1_x}, i,i)$, as distinct homotopy
  strings; of course they give rise to the same complex.
\item For each homotopy string $w$, there are unique
  trivial homotopy strings $1_x$ and $1_y$ such that the compositions $w
  1_x$ and $1_y w$ are defined.
\end{itemize}

We set up some notation and terminology.

\begin{notation}
Let $f = (f_r, \ldots, f_0) \in \Hom_{\sD}(P_w, P_v)$ be as illustrated below.
\[
\xymatrix@=1.4pc{
  \ar@{.}[r] & \xydot \ar[d]^{f_r}
  \ar@{-}[r]^-{w_n} & \cdots \ar@{-}[r]^-{w_l} & \xydot \ar[d]^{f_0} \ar@{.}[r] & \\
  \ar@{.}[r] & \xydot
  \ar@{-}[r]_-{v_{n'}} & \cdots \ar@{-}[r]_-{v_{l'}} & \xydot
  \ar@{.}[r] & \\
}
\]
The \emph{support of $f$} is the homotopy substring
$\support{f} \coloneqq \prod_{k=n}^l (w_k,i_k,j_k)$ of $w$.  For $l \leq i \leq n$ we will
also say that $f$ is supported at $P(\phi_w(i))$.  If $f$ has only one
non-zero component, say $f_0 \colon P(\phi_w(k)) \rightarrow
P(\phi_v(l))$, we say that $f$ is supported at $P(\phi_w(k))$.
\end{notation}

Recall the definition of antipath from Definition~\ref{def:antipath}.
A direct (resp. inverse) antipath $\theta$ is called \emph{maximal} if
there is no $a \in \Gamma_1$ such that $\theta(a,i,i+1)$ is a direct
antipath (resp. $(a,i,i-1)\theta$ is an inverse antipath).

\begin{remark}
  (Maximal) antipaths do not have to be finite; however, we cannot
  compose a homotopy string with an infinite antipath if this takes us
  outside $\sD$.  In particular, if $\theta = \theta_1 \theta_2 \cdots $ is an
  infinite direct antipath and $w$ is a finite homotopy string such
  that $w \theta_1$ is defined, then the composition $w \theta$ is not
  in $\sD$.  Similarly, if $\theta = \cdots \theta_2 \theta_1$ is an infinite
  inverse antipath and $\theta_1 w$ is defined, then the composition
  $\theta w$ is not in $\sD$.  Note that infinite
  antipaths correspond to oriented cycles in the quiver with `full
  relations'.
\end{remark}

\subsection{Algorithm for determining irreducible maps}
The algorithm is stated here for $\sK$.  We extend it to infinite
strings in Section \ref{sec:infinitealg}.  Let $w = \prod_{k=n}^1
(w_k, i_k,j_k) \in \strings$.  The strategy of the algorithm is as
follows: consider the identity map on $P_w$:
\[
\xymatrix@=1.2pc{
  \xydot \ar@{-}[r]^-{w_n} \ar@{=}[d] & \xydot \ar@{-}[r] \ar@{=}[d]
  & \cdots \ar@{-}[r] & \xydot
  \ar@{-}[r]^-{w_1} \ar@{=}[d] &\xydot \ar@{=}[d] \\
  \xydot \ar@{-}[r]^-{w_n} &\xydot \ar@{-}[r] & \cdots \ar@{-}[r] &
  \xydot
  \ar@{-}[r]^-{w_1} &\xydot \\
}
\]
We alter the map `from the left' in a minimal way to get a new map.
This gives a new string $w^+$ and a new map $f^+ \colon P_w \to
P_{w^+}$, which may each be zero. However, when they are nonzero, the
resulting map will turn out to be irreducible. This deals with one of
the two possible irreducible maps. To get the other, we alter the map
in a minimal way `from the right'. This gives a new string $w_+$ and a
new map $f_+ \colon P_w \to P_{w_+}$, which again may each be zero. At
least one of $f^+$ or $f_+$ will be nonzero.  We explicitly give the
algorithm which alters the map in a minimal way `from the left'. The
algorithm doing this `from the right' is dual.

\begin{algorithm} \label{alg:leftirred} The algorithm proceeds as
  follows, where we carry out each step in sequence unless instructed
  otherwise. The map $f'$ is defined in each step in
  Figure~\ref{fig:irred}.

\Step{1} If there is a direct homotopy letter $u$ such that $uw$ is a
homotopy string and $u$ is a maximal path, we set $w' = uw$, and go to
Step 8. 

\Step{2} Remove the longest direct antipath which is a left substring
of $w$. Write $\psi_w \coloneqq w_n \cdots w_{r+1}$ for this antipath.

\Step{3} If $(w_r,i_r,j_r)$ is inverse and there exists $a \in
\Gamma_1$ such that $w_r a \neq 0$, then set $w' = (w_r a,i_r,j_r)
\prod_{k=r-1}^1 (w_k,i_k,j_k)$, and go to Step 8.

\Step{4} If $(w_r,i_r,j_r)$ is inverse but there is no $a \in
\Gamma_1$ such that $w_r a \neq 0$, then set $w'
=\prod_{k=r-1}^{1}(w_k,i_k,j_k)$, and go to Step 9.

\Step{5} If $(w_r,i_r,j_r)$ is direct, then decompose $w_r = a a'$
with $a \in \Gamma_1$ and set $w' = (a',i_r,j_r) \prod_{k=r-1}^1
(w_k,i_k,j_k)$, and go to Step 8.

\Step{6} If $w$ is a direct antipath and there exists $a \in \Gamma_1$
with $t(a) = \phi_w(0)$ with $w_1 a =0$, then set $w'$ to be the
trivial homotopy string such that $ww'$ is defined, and go to Step 8.

\Step{7} Set $f'= f^+$ to be the zero map $P_w \to 0$ and terminate
the algorithm.

\Step{8} If there is a maximal inverse antipath $\theta =
\prod_{k=m}^1(\theta_k,i'_k,j'_k)$ such that $\theta w'$ is defined as
composition of homotopy strings, we set $w^+ = \theta w'$ and let
$f^+$ have the same components as $f'$ and terminate the algorithm.

\Step{9} Set $f^+ = f'$ and $w^+ = w'$.
\end{algorithm}

\begin{notation}
  We shall refer to the the dual algorithm that alters a homotopy
  string in a minimal way `from the right' as
  Algorithm~\ref{alg:leftirred}$'$; a prime will also be affixed to
  denote the dual of each of the corresponding steps, i.e. the dual of
  Step $i$ is Step $i'$.
\end{notation}

\begin{figure}\label{fig:irreducible}
\[
\begin{array}{ccc}
  \xymatrix@=1.4pc{
    & & &\xydot \ar@{-}[r]^-{w_n} \ar@{=}[d] & \xydot \ar@{-}[r] \ar@{=}[d] & \cdots
    \ar@{-}[r]^-{w_1} & \xydot \ar@{=}[d] \\
    \xydot & \ar[l]_{\theta_m} \cdots & \ar[l]_-{\theta_1} \xydot \ar[r]^{u} &\xydot
    \ar@{-}[r]^-{w_n} & \xydot \ar@{-}[r] & \cdots \ar@{-}[r]^-{w_1} &\xydot \\
  } & \quad &
  \xymatrix@=1.4pc{
    \xydot \ar[r]^-{w_n} &\cdots \ar[r] &\xydot \ar[d]^a & \xydot
    \ar[l]_-{w_r} \ar@{-}[r] \ar@{=}[d] & \cdots
    \ar@{-}[r]^-{w_1} & \xydot \ar@{=}[d] \\
    \xydot & \ar[l]_{\theta_m} \cdots & \ar[l]_-{\theta_1} \xydot & \xydot
    \ar[l]_-{w_r a} \ar@{-}[r] & \cdots
    \ar@{-}[r]^-{w_1} & \xydot \\
  }
  \\
  \text{Step 1} & \quad &
  \text{Step 3} \vspace{12pt}\\
  \xymatrix@=1.4pc{
    \xydot \ar[r]^-{w_n} &\cdots \ar[r] &\xydot & \xydot
    \ar[l]_-{w_r} \ar@{=}[d] & \xydot \ar@{=}[d] \ar@{-}[l]_-{w_{r-1}}
    \ar@{-}[r] & \cdots \ar@{-}[r]^-{w_1} &\xydot \ar@{=}[d] \\
    & & & \xydot
    & \xydot \ar@{-}[l]_-{w_{r-1}}
    \ar@{-}[r] & \cdots \ar@{-}[r]^-{w_1} &\xydot \\
  }
  & \quad &
  \xymatrix@=1.4pc{
    \xydot \ar[r]^-{w_n} &\cdots \ar[r] &\xydot \ar[d]^a \ar[r]^-{w_r} & \xydot
    \ar@{-}[r] \ar@{=}[d] & \cdots
    \ar@{-}[r]^-{w_1} & \xydot \ar@{=}[d] \\
    \xydot & \ar[l]_{\theta_m} \cdots & \ar[l]_-{\theta_1} \xydot  \ar[r]^-{a'} & \xydot 
    \ar@{-}[r] & \cdots
    \ar@{-}[r]^-{w_1} & \xydot \\
  }\\
  \text{Step 4}& \quad & \text{Step 5} \vspace{12pt}\\
  \xymatrix@=1.4pc{
    \xydot \ar[r]^-{w_n} &\cdots \ar[r]^-{w_1} &\xydot \ar[d]^a \\
    \xydot & \ar[l]_{\theta_m} \cdots & \ar[l]_-{\theta_1} \xydot \\
  } & \\
  \text{Step 6}& &\\
\end{array}
\]
\caption{Diagrams defining the homotopy strings $w'$ and maps $f'$
  produced in each step of Algorithm \ref{alg:leftirred}, with the
  maximal inverse antipath $\theta$ of Step 8 also added where
  appropriate.\label{fig:irred}}
\end{figure}
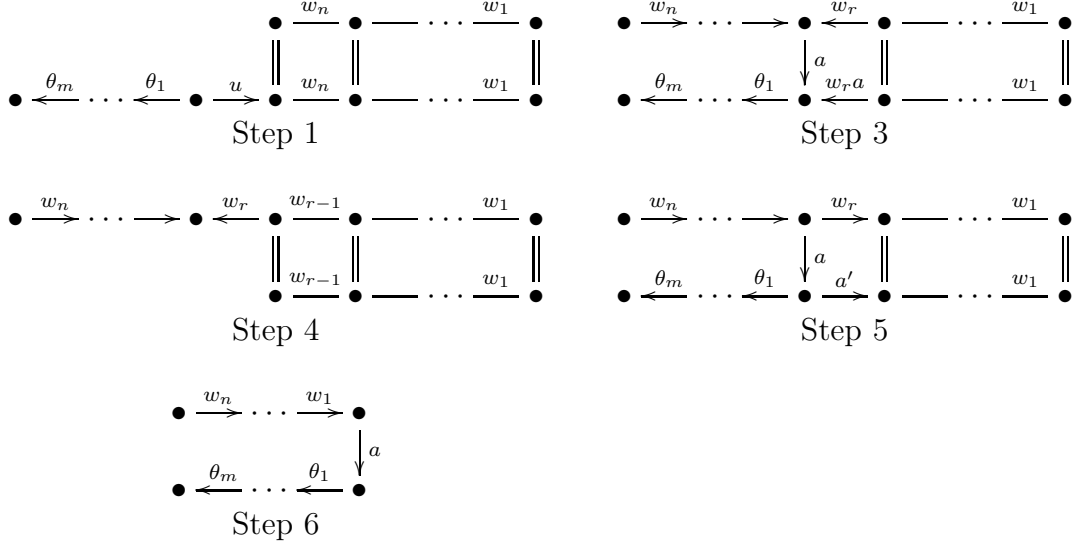

Note that Step 7 only occurs for homotopy strings that are either
trivial or antipaths.

\begin{remark}\label{rem:differentmaps} We highlight the following:
\begin{enumerate}[label=(\arabic*)]
\item The maps $f^+$ and $f_+$, when non-zero, are either graph maps
  or single maps.
\item If both $f^+$ and $f_+$ are non-zero, then the two maps are
  different at the level of chain maps.  If $w$ or $w'$ are trivial,
  then this is taken care of by the functions $S$ and $T$.
\item Whenever there is an $a \in \Gamma_1$ with $s(a) = x$, then Step 1
  or its `right' dual will occur when considering the trivial homotopy
  string $(1_x, i, i)$.  If there are two arrows $a$ and $b$ with
  $s(a) = x = s(b)$ then both Step 1 and its dual will occur.
\item A case analysis shows that if a maximal antipath $\theta$ is
  added in Step 8, then this antipath is finite.
\end{enumerate}
\end{remark}

\begin{theorem}\label{thm:irreducible}
  Let $w$ be a homotopy string, and let $f^+$ and $f_+$ be the outputs
  produced by Algorithm \ref{alg:leftirred} and its dual,
  respectively.  If $f^+$ (respectively $f_+$) is non-zero, then it is
  irreducible.
\end{theorem}

In the remainder of this section, we verify
Theorem~\ref{thm:irreducible} in the case that $f^+ \neq 0$. The case
$f_+ \neq 0$ is dual.

\subsection{Irreducible maps involving infinite homotopy
  strings} \label{sec:infinitealg} Let $w \in \stringsone \cup
\stringstwo$.  We extend the above algorithm by possibly extending
graph maps to infinite graph maps, and by adding the following step
between Step 8 and Step 9:

\Step{8.5} If there is a maximal, infinite inverse antipath $\theta =
\cdots (\theta_1,i'_1,j'_1)$ such that $(\theta_1,i'_1,j'_1) w'$ is defined as
composition of homotopy strings, we set $f^+$ to be the zero map $P_w
\rightarrow 0$ and terminate the algorithm.

As commented above, the condition of this step is impossible if $w$ is
a finite homotopy string.  Moreover, this step ensures that we never
go outside $\sD$.

\begin{lemma}
  If $w^+ \neq 0$ for some homotopy string $w$, then $w^+$ is in
  $\strings$ if and only if $w$ is in $\strings$.
\end{lemma}
\begin{proof}
  The `if' direction is already established.  Consider $v \in
  \stringsone$ such that $w \coloneqq{}^{\infty}v$ is left infinite.
  When we run the algorithm, we perform Step 2 and remove the left
  infinite part of $w$, which corresponds to a cyclic path $\rho$.  We
  will then be able to add an arrow $a$ in one of Step 3, Step 5 or
  Step 6 (take $a$ to be the arrow preceding $w_{r-1}$ in $\rho$).
  Next Step 8.5 will occur, since we can always add infinitely many
  copies of $\rho$, but this time as inverse homotopy letters.

  If $v \in \stringsone$ and $w \coloneqq v^{\infty}$, we will never
  remove the infinite part from $w$ as it consists of inverse homotopy
  letters.  Hence, if the output $w^+$ is non-zero, it is also in
  $\stringsone$.
\end{proof}

\begin{corollary}
  If $w$ is a two-sided infinite homotopy string, then there are no
  irreducible maps in $\sD$ with $P_w$ as source.
\end{corollary}

In the remainder of the section, we consider only irreducible maps in
$\sK$, but it is clear that the irreducible maps involving one-sided
homotopy strings will behave in the same way.

\subsection{Non-irreducible maps}

Before verifying Theorem~\ref{thm:irreducible}, we eliminate some maps
which can be easily seen to be non-irreducible.

\begin{proposition} \label{prop:not-irred}
Let $v,w \in \strings$ and $g \in \Hom_{\sK}(P_v,P_w)$. Then $g$ is {\bf not} irreducible if:
\begin{enumerate}[label=(\arabic*)]
\item \label{not-irred:double} $g$ is a double map;
\item \label{not-irred:graph} $g \in \graph{v}{w}$ is such that both
  endpoints are non-isomorphisms;
\item \label{not-irred:path} $g=(g_s,\ldots,g_0)$ is such that
  $g_s=ab$ with $a,b$ non-stationary paths in $\Gamma$;
\item \label{not-irred:dual} $g=(g_s,\ldots,g_0)$ is such that
  $g_0=ab$ with $a,b$ non-stationary paths in $\Gamma$.
\end{enumerate}
\end{proposition}

\begin{proof}
\ref{not-irred:double} 
 Write $g = (c,d)$ where $c$ and $d$ are paths in the quiver.  Then
  $g$ is, up to inverting $v$ and $w$, given by the diagram to the
  left, and the factorisation of $g$ is given by the diagram to the
  right. 
\[
\xymatrix@=1.2pc{
\xydot \ar@{-}[r] &\cdots \ar@{-}[r] &\xydot \ar[r]^{v_0} \ar[d]_{c} &
\xydot \ar[d]^{d} \ar@{-}[r] &\cdots \ar@{-}[r] & \xydot \\
\xydot \ar@{-}[r] &\cdots \ar@{-}[r] &\xydot \ar[r]^{w_0} & \xydot
\ar@{-}[r] & \cdots \ar@{-}[r] & \xydot
}
\qquad
\xymatrix@=1.2pc{
\xydot \ar@{-}[r] \ar@{=}[d] &\cdots \ar@{-}[r] &\xydot \ar[r]^{v_0} \ar@{=}[d] &
\xydot \ar[d]^{d} \ar@{-}[r] &\cdots \ar@{-}[r] & \xydot \\
\xydot \ar@{-}[r] &\cdots \ar@{-}[r] &\xydot \ar[r]^{v_0 d} \ar[d]^{c} &
\xydot \ar@{=}[d] \ar@{-}[r] &\cdots \ar@{-}[r] & \xydot \ar@{=}[d]\\
\xydot \ar@{-}[r] &\cdots \ar@{-}[r] &\xydot \ar[r]^{w_0} & \xydot
\ar@{-}[r] & \cdots \ar@{-}[r] & \xydot
}
\]

Cases \ref{not-irred:graph}, \ref{not-irred:path} and
\ref{not-irred:dual} are similar and can be verified by drawing the
corresponding diagrams, noting in \ref{not-irred:path} and
\ref{not-irred:dual} one needs to check only single and graph maps.
\end{proof}

As a consequence, no single map in the homotopy class of a double map
is irreducible.  We now examine single maps more closely, giving a
necessary condition for irreducibility, which is then used to
eliminate further possibilities for irreducible single maps.

\begin{proposition} \label{prop:single-necessary} If $g \colon P_v
  \rightarrow P_w$ is an irreducible single map, then $g$ is of the
  form
\begin{equation}
  \xymatrix@=1pc{
    \ar@{.}[r] & \xydot \ar[r]^{v_0}  & \xydot
    \ar[d]^a  \\
    & \xydot \ar@{.}[l] & \ar[l]_-{w_0} \xydot \\
  }
\label{eq:single-irred}
\end{equation}
up to inverting $w$ or $v$, where $a \in \Gamma_1$, and $v_0$ and/or
$w_1$ may be zero.  Moreover, $g$ is in a singleton homotopy class;
see Definition~\ref{def:singleton}.
\end{proposition}

\begin{proof}
  Suppose $g \in \single{v}{w}$ with non-zero support, by abuse of
  notation, denoted by $g\colon P \to P'$ with $P,P' \in
  \ind{\proj{\Lambda}}$. This situation is indicated below.
  \[
  \xymatrix@=1pc{
    \ar@{.}[r] & \xydot \ar@{-}[r]^{v_0}  & P
    \ar@{-}[r]^{v_1} \ar[d]^g & \xydot \ar@{.}[r]  & \\
    \ar@{.}[r] & \xydot \ar@{-}[r]_-{w_0} & P' \ar@{-}[r]_-{w_1} & \xydot
    \ar@{.}[r] & \\
  }
  \]
  One can easily check that in the following cases, $g$ is not
  irreducible:
  \begin{itemize}
  \item $v_0$ and $v_1$ are both non-zero, or precisely one is
    non-zero and it has $P$ as source.
  \item $w_0$ and $w_2$ are both non-zero, or precisely one is
    non-zero and it has $P'$ as target.
  \end{itemize}
  Indeed, if both conditions hold $g$ factors as two graph maps and
  one single map (with $g$ as its non-zero support), and if only one
  holds then $g$ factors as one graph map and one single map.
\end{proof}

\begin{corollary}\label{cor:single}
  Suppose $g \in \single{v}{w}$ has unfolded diagram taking the form
  \eqref{eq:single-irred} in
  Proposition~\ref{prop:single-necessary}. Then $g$ is {\bf not}
  irreducible if:
\begin{enumerate}[label=(\arabic*)]
\item \label{item:indirect} $v$ (or $w$) is not a uniformly oriented
  homotopy string;
\item \label{item:bigger-than-1} any homotopy letter of $v$ and $w$ is
  a path of length longer than $1$.
\end{enumerate}
\end{corollary}

\begin{proof}
The following diagrams indicate the factorisations:
  \[
  \xymatrix@=1.4pc{
    & \xydot \ar@{.}[l] & \ar[l]_-{v_k} \xydot \ar@{=}[d] \ar[r]^{v_{k-1}} & \cdots
    \ar[r] & \xydot \ar@{=}[d] \ar[r]^{v_0}  & \xydot
    \ar@{=}[d] \\
    & & \xydot \ar[r]^-{v_{k-1}} & \cdots \ar[r] & \xydot \ar[r]^-{v_0} & \xydot \ar[d]^g & & \\
    & & & & \xydot \ar@{.}[l] & \xydot \ar[l]_-{w_0} \\
  }
\qquad
  \xymatrix@=1.4pc{
    \cdots  \ar[r] & \xydot  \ar[r]^-{ab} \ar[d]^{a}& \xydot \ar@{=}[d] \ar[r]^{} & \cdots
    \ar[r] & \xydot \ar@{=}[d] \ar[r]^{}  & \xydot
    \ar@{=}[d] \\
    & \xydot \ar[r]^-{b} & \xydot \ar[r]^-{} & \cdots \ar[r]
    & \xydot \ar[r]^-{} & \xydot \ar[d]^g \\
    & & & & \xydot \ar@{.}[l] & \xydot \ar[l]_-{w_0} \\
  }
  \]
\end{proof}

\subsection{Proof of Theorem~\ref{thm:irreducible}}

We start by setting up the notation for the section. Throughout $w \in
\strings$, $w'$ and map $f' \colon P_w \to P_{w'}$ will be the
homotopy string and map output at Steps 1--7, and $w^+$ and $f^+
\colon P_w \to P_{w^+}$ the homotopy string and map output at Steps
7--9. Note that $f'$ differs from $f^+$ if and only if $w^+ = \theta
w'$, where $\theta = (\theta_m,j_m,i_m) \cdots (\theta_1,j_1,i_1)$ is
an inverse antipath.  Write $f^+ = (f_k,\ldots,f_0)$ where, using the
notation from Algorithm \ref{alg:leftirred}, $k \in \{0,r-1,r,n\}$.

\begin{lemma}\label{lem:antipath}
We have the following:
  \begin{enumerate}[label=(\arabic*)]
  \item $f^+$ is a well-defined map.
  \item $f'$ factors through $f^+$.
  \item For each $v \in \strings$ such that $vw'$ is defined and
    such that the components of $f'$ also determine a map $g \colon
    P_v \rightarrow P_{vw'}$, the map $g$ factors through $f^+$.
  \end{enumerate}
\end{lemma}

\begin{proof}This is easily verified.  Observe that if $f_k$ is the
  leftmost component of $f^+$, and if an antipath was added in Step 8
  such that $f_k\theta_1$ occurs in the diagram, then $f_k \theta_1 =
  0$. 
\end{proof}

The following lemma shows that if there is a graph map with source
$P_w$ satisfying certain criteria, then at least one of the maps $f^+$
and $f_+$ is a graph map.

\begin{lemma}\label{lem:biggest}
  Let $w$ be a non-trivial homotopy string.  If there exists $g \in
  \graph{w}{v}$ starting after the left endpoint of $w$ and stopping
  with an isomorphism at the right endpoint of $w$, then $f^+$ is a
  graph map and $g$ factors through $f^+$.
\end{lemma}

\begin{proof}
  It follows from the hypotheses of the lemma that $w$ is not an antipath,
  so neither Steps 6 nor 7 occur.  A straighforward case analysis
  shows that $g$ factors through $f^+$ even if $v$ is not a homotopy
  substring of $w^+$.
\end{proof}

\begin{lemma} \label{lem:substring} Suppose $f^+ \in \graph{w}{w^+}$
  and $g \colon P_w \to P_u$ is a map such that $\support{g}$ is a
  homotopy substring of $\support{f^+}$. If $g$ is not supported on
  the source of $f_k$, then $g$ factors through $f^+$.
\end{lemma}

\begin{proof}
  If $g \in \graph{w}{u}$, then since $f_{k-1},\ldots,f_0$ are
  isomorphisms, $\support{g}$ is also a homotopy substring of $w^+$. It
  is straightforward to check that the restriction $g' \colon P_{w^+} \to
  P_u$ of $g$ is also a graph map, and $g = g' f^+$. Similarly for $g
  \in \single{w}{u}$.
\end{proof}

There is a dual statement of Lemma \ref{lem:substring} for $f_+$.

\begin{lemma}\label{lem:nooverlap}
  Let $w \in \strings$ be such that $f^+ \in \graph{w}{w^+}$. If $g
  \colon P_w \to P_u$ is such that $\support{f^+}$ and $\support{g}$
  have no overlapping part, then $f_+ \neq 0$ and $g$ factors through
  $f_+$.
\end{lemma}

\begin{proof}
Write $g = (g_s, \ldots, g_0)$.
Recall the notation from Algorithm~\ref{alg:leftirred}. Only the outputs of Steps 1, 3, 4 and 5 of Algorithm~\ref{alg:leftirred} are graph maps.
We consider these possible termination points in turn.
\begin{itemize}
\item At Step 1: Algorithm \ref{alg:leftirred} cannot output $f^+$ at
  Step 1 under these conditions, since $\support{f^+} = w$.
\item At Steps 3 or 5: in these cases $\support{g}$ is a substring of
  $\psi_w$. Thus $w$ contains direct letters and the duals of Steps 6
  and 7 cannot occur. In particular, the dual of
  Algorithm~\ref{alg:leftirred} outputs a graph map $f_+$. Now apply
  Lemma~\ref{lem:substring}.
\item At Step 4: the rightmost possible support of $g_0$ is
  $P(\phi_w(r))$. If this is not attained, then $\support{g}$ is a
  substring of $\psi_w$ and one argues as in Step 3 and 5 above.

  Suppose the support of $g_0$ is $P(\phi_w(r))$. If $g \in
  \single{w}{u}$ then $w_n = w_r$ and $g$ factors through $f_+$: Steps
  $1'$, $3'$ and $4'$ yield a common substring of $\support{f_+}$ and
  $\support{f^+}$, whence we apply
  Lemma~\ref{lem:substring}; and, if Step $6'$
  occurs then we must have $g_0 = aa'$ where $a$ is the non-zero
  component of $f_+$, and this gives the desired conclusion.  In
  particular Step $7'$ never occurs and $f_+ \neq 0$.

  If $g \in \graph{w}{u}$ we have isomorphisms in all components to
  the left of $g_0$, and the conclusion follows by
  Lemma~\ref{lem:biggest}.
\end{itemize}
Thus, $f_+ \neq 0$ and $g$ factors through $f_+$, as desired.
\end{proof}

We are now ready to prove Theorem~\ref{thm:irreducible}.

\begin{proof}[Proof of Theorem \ref{thm:irreducible}]
  Let $w\in \strings$ and $f^+ = (f_k, \ldots, f_0)$ and $w^+$ be the
  outputs of Algorithm \ref{alg:leftirred}. The proof is a case
  analysis. We start with some generalities. Let $g \colon P_w \to
  P_v$ be another candidate for an irreducible map and note the
  following.
\begin{itemize}
\item Lemmas~\ref{lem:substring} and \ref{lem:nooverlap} say that any
  map $g$ not supported on the source of $f_k$ factors through $f^+$
  or $f_+$, and so is not irreducible.
\item Proposition~\ref{prop:not-irred}\ref{not-irred:double} says that
  $g$ must be a graph map or a single map.
\item Proposition
  \ref{prop:not-irred}\ref{not-irred:path}\ref{not-irred:dual} and
  Corollary \ref{cor:single}\ref{item:bigger-than-1} say that all
  components of $g$ which are not isomorphisms are arrows.
\end{itemize}
We therefore write $g=(g_s,\ldots,g_0,\ldots,g_{-t})$, where $g_0$ is
the component supported on the source of $f_k$; if $g$ is a single map
then $g=g_0$.  If $g \in \graph{w}{v}$, we fix the orientation of $v$
such that the common substring of $v$ and $w$ has the same
orientation.

We treat the three possible cases for each step of
Algorithm~\ref{alg:leftirred} which outputs $f^+$.

\smallskip
\Listcase{1} {\it The map $g$ is a graph map and $g_0$ is an arrow.}

\Listcase{2} {\it The map $g$ is a graph map and $g_0$ is an isomorphism.}

\Listcase{3} {\it The map $g$ is a single map, in which case $g_0$
  must be an arrow.} 
\smallskip

\noindent It may be useful for the reader to keep
Figure~\ref{fig:irred} on page \pageref{fig:irred} in mind.

\smallskip
\noindent{\bf The map $f^+$ is output at Step 1.} 
In this case, $g_0$ is the leftmost component of $g$.

\Case{1}  Clearly, $g$ factors through $f^+$ since $ug_0 = 0$.

\Case{2} If $t < n$ then we simply apply Lemma~\ref{lem:biggest} to
get $g$ factoring through $f_+$. If $t = n$ and $g_{-n}$ is an
isomorphism, then $g$ factors through $f^+$ by Lemma
\ref{lem:antipath}. If $t=n$ and $g_{-n}$ is an arrow, then it is clear
that $g$ factors through $f_+$.

\Case{3} For $g=g_0$ to be irreducible we need $w_n$ to be inverse or
trivial by Proposition~\ref{prop:single-necessary}. If $w_n$ is
inverse, $g$ cannot factor through $f^+$ since $u g \neq 0$. If $f_+$
is a graph map then $g$ factors through $f_+$ by
Lemma~\ref{lem:substring}. If $f_+$ is a single map, then it is the
output of Step $6'$, in which case $g_0$ is the non-zero component of
$f_+$ and the factorisation follows. If $w_n$ is trivial then $w$ is
trivial and $P_w$ is a stalk complex. It can then be easily checked
that $g$ factors through $f^+$ or $f_+$.

\smallskip
\noindent{\bf The map $f^+$ is output at Step 3.}
Here $w$ is non-trivial and $g_0$ need not be the leftmost non-zero
component of $g$.

\Case{1} Here $g_0$ is either the leftmost or rightmost non-zero
component of $g$. In the former case $w_r g_0 \neq 0$ and $g_0 = a$
(in the notation of Algorithm~\ref{alg:leftirred}). By
Proposition~\ref{prop:not-irred}\ref{not-irred:graph}, we may assume
that $g_{-t}$ is an isomorphism. Now $g$ factors through $f^+$ by
Lemma~\ref{lem:antipath} if $t<r$. If $t=r$, $f_+$ is a graph map
(argue as in Lemma \ref{lem:biggest}), and the conclusion follows by
Lemma~\ref{lem:substring}.

\Case{2} In this case $g$ does not factor through $f^+$. The
components $g_{-1},\ldots, g_{-t+1}$ are isomorphisms; $g_{-t}$ may be
an arrow. The $g_{s},\ldots,g_{1}$ are all isomorphisms if $\psi_w$ is
non-trivial, and zero otherwise. If $f_+$ is a graph map then $g$
factors through $f_+$ by Lemma \ref{lem:biggest}. If $f_+$ is not a
graph map, then $\prod_{k=r}^1(w_k,i_k,j_k)$ is an inverse antipath,
which implies $g_{-t}$ is an isomorphism. It follows that $w = u$
since adding anything to the endpoints forces either $f^+$ to be
output in Step 1, or a graph map $f_+$ in Step $1'$, whence $g$ is an
isomorphism.

\Case{3} Apply the same argument as when $g_0$ was the rightmost
non-zero component of $g$ in Case 1.

\smallskip
\noindent{\bf The map $f^+$ is output at Step 4.} 
Here Cases 1 and 3 are straightforward. For Case 1 apply
Lemma~\ref{lem:substring}. For Case 3 it is immediate that $g$ factors
through $f^+$.

\Case{2} If $g_0$ is the leftmost non-zero component then by
previously given arguments, $g$ factors through $f^+$. So assume not:
we must have that $g_{1}$ is an isomorphism, for otherwise $f^+$ would
have been output at Step 3. Applying the argument as in Case 2 above,
$g_{s}, \ldots, g_{1}$ are each isomorphisms. If $t<r-1$ then $g$
factors through $f_+$ by Lemma~\ref{lem:substring}. If $t=r-1$, $g$ is
either an isomorphism or else Algorithm~\ref{alg:leftirred}$'$ outputs
$f_+$ at Step $1'$, whence $g$ clearly factors through $f_+$.

\smallskip
\noindent{\bf The map $f^+$ is output at Step 5.} Here $f_+$ cannot be
output at Step $6'$ of Algorithm~\ref{alg:leftirred}$'$ since $w_r$ is
direct, so $f_+$ is a graph map. For Case 2, argue as in
Case 2 above. Case 3 is again a straighforward verification.

\Case{1} If $g_0$ is the leftmost non-zero component, then $g_0 = a$
(in the notation of Algorithm~\ref{alg:leftirred}) factors through
$f^+$. If $g_0$ is the rightmost non-zero component then use
Lemma~\ref{lem:biggest}.

\smallskip
\noindent{\bf The map $f^+$ is output at Step 6.} Note that the direct
antipath $w$ may be trivial.

\Case{1 and 3} If $g_0 \neq a$ (in the notation of
Algorithm~\ref{alg:leftirred}) then in the graph map case $g$ factors
through $f^+$ by Lemma~\ref{lem:biggest}. In the single map case, $w$
is a trivial homotopy string and $g$ factors through $f_+$, which is
output either at Step $1'$ or Step $6'$ of
Algorithm~\ref{alg:leftirred}$'$. Similarly when $g_0 = a$.

\Case{2} Here, the isomorphism $g_0$ is dragged down the entire
homotopy string so that $v$ has the form $v = \cdots c w d \cdots$ for
some homotopy letters $c$ and $d$. The homotopy letter $c$ cannot be
direct ($f^+$ would be output at Step 1) nor inverse ($g$ would not be
a map). Similarly, $d$ is not direct. Hence $v = wd \cdots$. If $d$ is
trivial then $g$ is an isomorphism. If $d$ is inverse then $f_+$ is
output at Step $1'$ and $g$ clearly factors through $f_+$.

\smallskip
\noindent{\bf The map $f^+=0$ is output at Step 7.} One now applies
the dual arguments for $f_+$ to get the required factorisation through
$f_+$.
\end{proof}

\section{Application: Discrete derived categories}

For the definition and background on discrete derived categories we
refer the reader to \cite{BGS,BPP1,Vossieck}. Derived-discrete
algebras are derived equivalent to either path algebras of
simply-laced Dynkin quivers or the bound path algebra $\Lambda(r,n,m)$
defined in Figure~\ref{fig:tennis}; when we refer to discrete derived
categories, we shall always mean $\Db(\Lambda(r,n,m))$.

In this section, we recover the universal Hom-space dimension bound
described in \cite{BPP1} when $\Lambda(r,n,m)$ is of finite global
dimension and extend it to the case $\Lambda(r=n,n,m)$, which has
infinite global dimension. Recall from \cite{BM} that $\Lambda(r,n,m)$
has no homotopy bands. From now on $\Lambda \coloneqq \Lambda(r,n,m)$
and $\strings(\Lambda)$ will denote homotopy strings over
$\Lambda$.
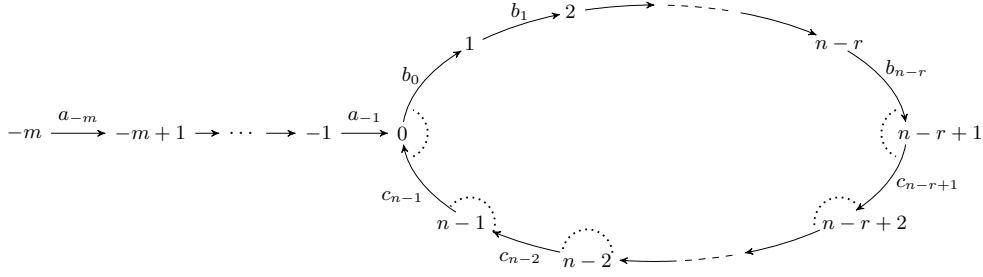
\begin{figure}
\[
\scalebox{0.85}{
\begin{tikzpicture}[xscale=1.3]
  \node (-m) at (-4.5,0) [tinyvertex] {$-m$};
  \node (-m+1) at (-3,0) [tinyvertex] {$-m+1$};
  \node (-m+2) at (-1.9,0) [tinyvertex] {$\cdots$};
  \node (-1) at (-1,0) [tinyvertex] {$-1$};
  \node (0) at (0,0) [tinyvertex] {$0$};
  \draw[->] (-m) to node[tinyvertex, above] {$a_{-m}$} (-m+1);
  \draw[->] (-m+1) to (-m+2);
  \draw[->] (-m+2) to (-1);
  \draw[->] (-1) to node[tinyvertex, above] {$a_{-1}$}(0);
  \node (1) at (0.8,1.4) [tinyvertex] {$1$};
  \node (n-1) at (0.7,-1.4) [tinyvertex] {$n-1$};
  \node (2) at (2,1.9) [tinyvertex] {$2$};
  \node (n-2) at (2.2,-2) [tinyvertex] {$n-2$};
  \node (n-r) at (5.2,1.4) [tinyvertex] {$n-r$};
  \node (n-r+2) at (5.5,-1.4) [tinyvertex] {$n-r+2$};
  \node (n-r+1) at (6.4,0) [tinyvertex] {$n-r+1$};
  \draw[->] (3,0) [partial ellipse=40:5:3cm and 2cm]; 
  \draw[->] (3,0) [partial ellipse=-5:-37:3cm and 2cm]; 
  \draw[->] (3,0) [partial ellipse=175:140:3cm and 2cm]; 
  \draw[->] (3,0) [partial ellipse=133:112:3cm and 2cm]; 
  \draw[->] (3,0) [partial ellipse=218:185:3cm and 2cm]; 
  \draw[->] (3,0) [partial ellipse=248:230:3cm and 2cm]; 
  \draw[->] (3,0) [partial ellipse=106:90:3cm and 2cm];
  \draw[dashed] (3,0) [partial ellipse=87:73:3cm and 2cm];
  \draw[->] (3,0) [partial ellipse=70:50:3cm and 2cm];
  \draw[->] (3,0) [partial ellipse=-49:-69:3cm and 2cm];
  \draw[dashed] (3,0) [partial ellipse=-72:-84:3cm and 2cm];
  \draw[->] (3,0) [partial ellipse=-85:-98:3cm and 2cm];
  \node at (0.1,0.9) [tinyvertex] {$b_0$};
  \node at (1.4,1.9) [tinyvertex] {$b_1$};
  \node at (6,1) [tinyvertex] {$b_{n-r}$};
  \node at (6.25,-0.8) [tinyvertex] {$c_{n-r+1}$};
  \node at (0,-1) [tinyvertex] {$c_{n-1}$};
  \node at (1.4,-1.95) [tinyvertex] {$c_{n-2}$};
  \draw (0) [dotted, thick, partial ellipse=-65:65:0.3cm and 0.4cm];
  \draw (6,0) [dotted, thick, partial ellipse=115:245:0.3cm and 0.4cm];
  \draw (5.2,-1.4) [dotted, thick, partial ellipse=52:195:0.3cm and 0.4cm];
  \draw (0.8,-1.4) [dotted, thick, partial ellipse=-20:140:0.3cm and 0.4cm];
  \draw (2.2,-1.9) [dotted, thick, partial ellipse=5:170:0.3cm and 0.4cm];
\end{tikzpicture}}
\]
\caption{The bound quiver defining $\Lambda(r,n,m)$.}\label{fig:tennis}
\end{figure}

A \emph{subword} of a homotopy string $w$ is defined in the obvious
fashion: the left- and rightmost homotopy letters of the subword may
be (incomplete) substrings of the corresponding letters of $w$. We now
describe all the homotopy strings for a discrete derived
category. Note that, when $r=n$, by our labelling convention there are
no `$b$' arrows. 

\begin{lemma}
Consider the following homotopy strings:
\[ 
\xymatrix@C=2.8pc{
w_k \colon 
  & \xydot 
  & {\circ} \ar[l]_{a_{-1}\ldots a_{-m}} \ar[r]^-{v_k} 
  &   \circ \ar[r]^{c_{n-1}} 
  &  \cdots \ar[r]^{c_{n-r+1}}  
  & \xydot \ar[r]_{b_{n-r}\ldots b_0 a_{-1}\ldots a_{-m}} 
  & \xydot
  &
\\ 
w \colon
  & \cdots \ar[r]^{c_{n-1}}
  & \xydot \ar[r]^{c_{n-2}}
  & \cdots \ar[r]^{c_0}
  & \xydot \ar[r]^{c_{n-1}}
  & \xydot \ar[r]^{c_{n-2}}
  & \cdots \ar[r]^{c_1}
  & \xydot \ar[r]_{c_0 a_{-1}...a_{-m}} & \xydot }
\] 
where $v_k$ is the $k$-fold concatenation of  $\xymatrix@C=2.8pc{ \xydot
  \ar[r]^-{c_{n-1}} & \cdots \ar[r]^-{c_{n-r+1}} & \xydot
  \ar[r]^-{b_{n-r}\ldots b_0} & \xydot }$. Then:
\begin{enumerate}[label=(\arabic*)]
\item if $r < n$, all homotopy strings are (shifted) copies of
  subwords of the $w_k$ for $k \geq 0$;
\item if $r = n$, all homotopy strings are (shifted) copies of
  subwords of $w$ and $w_k$ for $k \geq 0$.
\end{enumerate}
\end{lemma}

\begin{lemma} \label{obs:one-map} If $r >1$ we have
  $\hom_\Lambda(P(i), P(j)) \leq 1$ for all $-m \leq i,j < n$. If
  $r=1$ we have additionally $\hom_\Lambda(P(0),P(j)) = 2$ for all $-m
  \leq j \leq 0$.
\end{lemma}

\begin{lemma}\label{lem:dim1}
  Suppose $v,w \in \strings(\Lambda)$ and consider the following
  unfolded diagram:
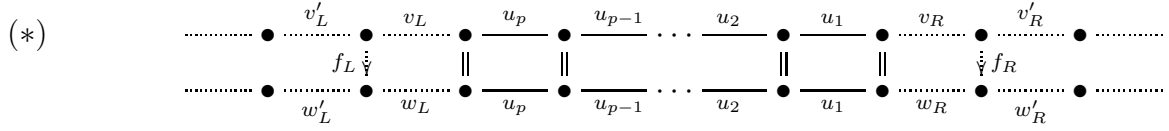
\begin{equation}
\label{unfolded}
\tag{$*$} 
\xymatrix@R=3mm{ 
 & \ar@{.}[r]    & \xydot  \ar@{.}[r]^{v_L'} & \xydot \ar@{.>}[d]_{f_L} \ar@{.}[r]^{v_L} & \xydot  \ar@{=}[d] \ar@{-}[r]^{u_p} & \xydot \ar@{=}[d] \ar@{-}[r]^{u_{p-1}} & \cdots \ar@{-}[r]^{u_2}  & \xydot \ar@{=}[d] \ar@{-}[r]^{u_1} & \xydot \ar@{=}[d] \ar@{.}[r]^{v_R} & \xydot \ar@{.>}[d]^{f_R} \ar@{.}[r]^{v_R'} &  \xydot \ar@{.}[r]   &  \\
 & \ar@{.}[r]     & \xydot \ar@{.}[r]_{w_L'} & \xydot \ar@{.}[r]_{w_L}                  & \xydot \ar@{-}[r]_{u_p}                             & \xydot \ar@{-}[r]_{u_{p-1}}            & \cdots \ar@{-}[r]_{u_2}  & \xydot \ar@{-}[r]_{u_1}            & \xydot \ar@{.}[r]_{w_R}             & \xydot \ar@{.}[r]_{w_R'}                  & \xydot \ar@{.}[r]  &  
}
\end{equation}
\begin{enumerate}[label=(\arabic*)] 
\item \label{qgraph} Suppose \eqref{unfolded} represents a quasi-graph
  map $P_v \to \Sigma^{-1}P_w$.  If $v_L$ is non-zero, then one of
  $v_L'$ or $w_L'$ is zero.

\item \label{graph} Suppose \eqref{unfolded} represents a graph map
  $P_v \to P_w$. Then either
\begin{enumerate}[label=(\roman*)]
\item $f_L \neq 0$ and one of $v_L'$ or $w_L'$ is zero; or
\item $f_L = 0$ and  if $v_L$ is non-zero then $v_L =
  (a_{-1}...a_{-i})^{-1}$ for some $1 \leq i \leq m$.
\end{enumerate}
\end{enumerate}
Dual statements hold for the right-hand end of the diagram.
\end{lemma}

\begin{proof} 
  Recall from Definition \ref{def:quasigraph} that if $v_L$ is not
  zero then $v_L \neq w_L$.  We note that the only ways that $v_L$ and
  $w_L$ can differ is if one of them is zero or if one (or both) of
  them is a subpath of $b_{n-r}\ldots b_0 a_{-1} \ldots a_{-m}$, and
  it is clear that in all cases one of $v_L'$ or $w_L'$ is zero.  This
  shows \ref{qgraph}; \ref{graph} is similar.
\end{proof}

The upshot of Lemma~\ref{lem:dim1} is that any graph map $P_v \to P_w$
or quasi-graph map $P_v \to \Sigma^{-1}P_w$ spans every degree where
$P_w$ and $P_v$ are both non-zero.
 
\begin{theorem}
  Suppose $v,w \in \strings(\Lambda)$. If $r > 1$ then
  $\hom_{\Db(\Lambda)}(P_v, P_w) \leq 1$. If $r = 1$ then
  $\hom_{\Db(\Lambda)}(P_v,P_w) \leq 2$.
\end{theorem}

\begin{proof}
  For $r > 1$, observe that Lemma~\ref{obs:one-map} combines with
  Lemmas~\ref{lemma: basis maps determined} and \ref{lem:dim1} to give
  $\hom_{\Db(\Lambda)}(P_v,P_w)$ in the cases that there is a graph
  map $P_v \to P_w$ or a non-singleton homotopy class $P_v \to P_w$.

  For $r\geq 1$, we claim that if there is a single or double map $f
  \colon P_v \to P_w$ such that $\mcH(f)$ is a singleton homotopy
  class, then $\hom_\sD(P_v, P_w) = 1$. If $f$ is a single map, a case
  analysis reveals that, of the options presented in Definition
  $\ref{def:singleton}$, only $(i)$ could arise.  Clearly, there can
  be no other basis maps $P_v \to P_w$. Similarly, by considering all
  of the possible cases where $f$ is a double map, we find that either
  $v$ or $w$ is a homotopy string of length $1$ and that in each of
  these cases there is no other possible basis map.

  When $r=1$, the only way there can be more than one basis map $P_v
  \to P_w$ is if there is a graph map and a single map supported in
  the same degree, both the graph map and the homotopy class
  containing this single map will span the rest of the string, by
  Lemmas~\ref{lemma: basis maps determined} and \ref{lem:dim1}.
\end{proof}

\begin{example}
The following example shows that the upper bound can be attained.  Let
$\Lambda = \Lambda(1,1,3)$ and consider the homotopy strings $v =
(b_2b_1b_0, 2, 3)(b_2b_1b_0a_{-1}, 3, 4)$ and $w =
(a_{-1},2,1)(b_2b_1b_0, 1, 2)(b_2b_1b_0, 2, 3)(b_2b_1b_0a_{-1}, 3,
4)$.  Pictured below, from left to right: The
algebra $\Lambda(1,1,3)$, a graph map $P_v \to P_w$, and a
quasi-graph map $P_v \to \Sigma^{-1}P_w$.
\[
\begin{array}{cc}
\scalebox{0.9}{
\begin{tikzpicture}[scale=1,baseline=0.5cm]
  \node (-1) at (-1.5,0) [tinyvertex] {$-1$};
  \node (0) at (0,0) [tinyvertex] {$0$};
  \draw[->] (-1) to node[tinyvertex, above] {$a_{-1}$}(0);
  \node (1) at (1.3,0.7) [tinyvertex] {$1$};
  \node (2) at (1.3,-0.7) [tinyvertex] {$2$};
  \draw[->] (0.8,0) [partial ellipse=165:60:0.8cm]; 
  \draw[->] (0.8,0) [partial ellipse=40:-43:0.8cm]; 
  \draw[->] (0.8,0) [partial ellipse=-65:-165:0.8cm]; 
  \node at (0.2,0.8) [tinyvertex] {$b_0$};
  \node at (1.8,0) [tinyvertex] {$b_1$};
  \node at (0.3,-0.8) [tinyvertex] {$b_2$};
  \draw (0) [dotted, thick, partial ellipse=-65:65:0.3cm and 0.3cm];
\end{tikzpicture}
}
&\scalebox{0.9}{
\xymatrix@R=1.5pc{ & & \xydot \ar[r]^{b_2b_1b_0} \ar@{=}[d] & \xydot \ar@{=}[d] \ar[r]^{b_2b_1b_0a_{-1}} & \xydot \ar@{=}[d] \\
\xydot & \xydot \ar[l]^{a_{-1}} \ar[r]_{b_2b_1b_0} & \xydot \ar[r]_{b_2b_1b_0} & \xydot \ar[r]_{b_2b_1b_0} & \xydot
} \qquad
\xymatrix@R=1.5pc{  & \xydot \ar[r]^{b_2b_1b_0} \ar@{=}[d] & \xydot \ar@{=}[d] \ar[r]^{b_2b_1b_0a_{-1}} & \xydot  & \\
\xydot & \xydot \ar[l]^{a_{-1}} \ar[r]_{b_2b_1b_0} & \xydot
\ar[r]_{b_2b_1b_0} & \xydot \ar[r]_{b_2b_1b_0} & \xydot
}}
\end{array}
\] 
\end{example}


\end{document}

%% file: header.tex
\usepackage{graphicx}

\usepackage[all]{xy}
\usepackage{placeins}
\usepackage{enumitem}
\usepackage{amssymb}
\usepackage{latexsym}
\usepackage{amsmath}
\usepackage{mathrsfs}
\usepackage{array,booktabs}
\usepackage{verbatim}

\usepackage[notref, notcite]{showkeys}  
\usepackage{color}

\usepackage{hyperref}
\newcommand{\harxiv}[1]{ \href{http://arxiv.org/abs/#1}{\texttt{arXiv:#1}}}
\newcommand{\hyref}[2]{ \hyperref[#2]{#1~\ref*{#2}} }

\newcommand{\coloneqq}{\mathrel{\mathop:}=}

\usepackage{fullpage}

\theoremstyle{plain}
\newtheorem{theorem}{Theorem}[section]

\newtheorem{lemma}[theorem]{Lemma}
\newtheorem{corollary}[theorem]{Corollary}
\newtheorem{proposition}[theorem]{Proposition}

\newtheorem{introtheorem}{Theorem}

\theoremstyle{definition}
\newtheorem{remark}[theorem]{Remark}
\newtheorem{remarks}[theorem]{Remarks}
\newtheorem{example}[theorem]{Example}
\newtheorem{definition}[theorem]{Definition}

\newtheorem{notation}[theorem]{Notation}

\newtheorem{convention}[theorem]{Convention}
\newtheorem*{runningexample}{Running Example}
\newtheorem{algorithm}[theorem]{Algorithm}

\newcommand{\Step}[1]{\medskip\noindent\emph{Step #1:}}
\newcommand{\Case}[1]{\smallskip\noindent\emph{Case #1:}}
\newcommand{\Listcase}[1]{\noindent\emph{Case #1:}}
\newcommand{\Casecolon}[1]{\medskip\noindent\emph{Case: #1.}}

\newcommand{\None}{$(${\bf N1}$)$}
\newcommand{\Ntwo}{$(${\bf N2}$)$}
\newcommand{\Nthree}{$(${\bf N3}$)$}
\newcommand{\LGone}{$(${\bf LG1}$)$}
\newcommand{\LGtwo}{$(${\bf LG2}$)$}
\newcommand{\RGone}{$(${\bf RG1}$)$}
\newcommand{\RGtwo}{$(${\bf RG2}$)$}

\newcommand{\Lone}{$(${\bf L1}$)$}
\newcommand{\Ltwo}{$(${\bf L2}$)$}
\newcommand{\Rone}{$(${\bf R1}$)$}
\newcommand{\Rtwo}{$(${\bf R2}$)$}
\newcommand{\D}{$(${\bf D}$)$}

\newcommand{\kk}{{\mathbf{k}}}

\renewcommand{\setminus}{\backslash}

\DeclareMathOperator{\Hom}{\mathsf{Hom}}

\DeclareMathOperator{\Ext}{\mathsf{Ext}}
\DeclareMathOperator{\id}{{\mathsf{id}}}

\DeclareMathOperator{\Rad}{\mathsf{Rad}}
\newcommand{\ind}[1]{\mathsf{ind}(#1)}
\renewcommand{\mod}[1]{\mathsf{mod}(#1)}
\newcommand{\stmod}[1]{\underline{\mathsf{mod}}(#1)} 
\newcommand{\proj}[1]{\mathsf{proj}(#1)}

\renewcommand{\hom}{\mathrm{hom}}
\DeclareMathOperator{\rad}{\mathrm{rad}}

\DeclareMathOperator{\Span}{\mathsf{Span}}
\newcommand{\rL}{\hat{\Lambda}} 

\renewcommand{\graph}[2]{\mcG_{#1, #2}}
\newcommand{\single}[2]{\mcS_{#1, #2}}
\newcommand{\double}[2]{\mcD_{#1, #2}}
\newcommand{\basisC}[2]{\mcB^{\sC}_{#1, #2}}

\newcommand{\basisD}[2]{\mcB^{\sD}_{#1,#2}}

\newcommand{\singleone}[2]{\mcS^1_{#1, #2}}
\newcommand{\doubleone}[2]{\mcD^1_{#1, #2}}
\newcommand{\quasi}[2]{\mcQ_{#1,#2}}

\newcommand{\strings}{\mathsf{St}}
\newcommand{\bands}{\mathsf{Ba}}

\newcommand{\stringsone}{\mathsf{St}_1}
\newcommand{\stringstwo}{\mathsf{St}_2}

\newcommand{\support}[1]{\mathsf{supp}(#1)}

\newcommand{\homotopy}[1]{\mcH(#1)}

\newcommand{\homotopyband}[1]{\widetilde{\mcH}(#1)}
\newcommand{\homotopylayers}[3]{\widetilde{\mcH}^{(#1,#2)}(#3)}

\newcommand{\sC}{\mathsf{C}}
\newcommand{\sD}{\mathsf{D}}

\newcommand{\sK}{\mathsf{K}}

\DeclareMathAlphabet{\mathpzc}{OT1}{pzc}{m}{it}

\newcommand{\bN}{\mathbb{N}}

\newcommand{\mcB}{\mathcal{B}}
\newcommand{\mcC}{\mathcal{C}}
\newcommand{\mcD}{\mathcal{D}}

\newcommand{\mcG}{\mathcal{G}}
\newcommand{\mcH}{\mathcal{H}}
\newcommand{\mcI}{\mathcal{I}}

\newcommand{\mcQ}{\mathcal{Q}}

\newcommand{\mcS}{\mathcal{S}}

\newcommand{\Db}{\sD^b}
\newcommand{\Kb}{\sK^b}
\newcommand{\Kbp}[1]{\Kb(\proj #1)}
\newcommand{\KbpL}{\Kbp{\Lambda}}

\newcommand{\CL}{\sC^{-,b}(\proj \Lambda)}
\newcommand{\DbL}{\Db(\Lambda)}
\newcommand{\KpL}{\sK^{-,b}(\proj \Lambda)}

\newcommand{\cxP}{P^{\bullet}}
\newcommand{\cxQ}{Q^{\bullet}}

\newcommand{\xydot}{{\bullet}}
\newcommand{\arr}{\ar@{-}[r]}
\newcommand{\ard}{\ar@{=}[d]}

\newcommand{\dd}{\partial}

\renewcommand{\path}{\leadsto}

\newcommand{\into}{\hookrightarrow}

\newcommand{\bij}{\stackrel{1-1}{\longleftrightarrow}}

\newenvironment{pmat}{\left( \begin{smallmatrix}}{\end{smallmatrix} \right)}

\newcommand{\smxy}[1]{{\text{\tiny$#1$}}}

\renewcommand{\phi}{\varphi}
\renewcommand{\epsilon}{\varepsilon}

\usepackage{tikz}
\usetikzlibrary{snakes}
\usetikzlibrary{shapes.geometric,positioning}
\usetikzlibrary{arrows,decorations.pathmorphing,decorations.pathreplacing}
\tikzset{vertex/.style={circle,fill=black,inner sep=1pt,outer sep=2pt},
         tinyvertex/.style={font=\scriptsize,minimum size=6pt},
         smallvertex/.style={inner sep=1pt, font=\small},
         >=stealth',
         leadsto/.style={-angle 90,decorate,decoration=snake,very thick},
         cut/.style={decorate,decoration=saw,very thick}}
\tikzset{
    partial ellipse/.style args={#1:#2:#3}{
        insert path={+ (#1:#3) arc (#1:#2:#3)}
    }
}
